\documentclass[a4paper]{article}
\usepackage{fullpage}
\usepackage{amsmath,amsfonts,amsthm,amssymb}
\usepackage{color,xcolor}
\usepackage{ifpdf}
\usepackage{psfrag}
\usepackage{graphicx,graphics,subfigure}
\usepackage{url}
\usepackage{hyperref}
\usepackage{todonotes}

\newtheorem{theorem}{Theorem}[section]

\newtheorem{lemma}[theorem]{Lemma}
\newtheorem{proposition}[theorem]{Proposition}
\newtheorem{remark}[theorem]{Remark}

\newcommand{\dpar}[2]{\dfrac{\partial #1}{\partial #2}}

\renewcommand{\P}{\mathbb P}

\newcommand{\hbbf}{\hat{\mathbf{f}}}

\newcommand{\bbf}{\mathbf{f}}
\newcommand{\bg}{\mathbf{g}}

\newcommand{\bu}{\mathbf{u}}
\newcommand{\bv}{\mathbf{v}}
\newcommand{\bw}{\mathbf{w}}
\newcommand{\bx}{\mathbf{x}}

\newcommand{\bA}{\mathbf{A}}

\newcommand{\bF}{\mathbf{F}}
\newcommand{\bG}{\mathbf{G}}

\usepackage{float}
\begin{document}
\title{An arbitrarily high order and asymptotic preserving kinetic scheme in compressible fluid dynamic 
}
\author{R\'emi Abgrall and Fatemeh Nassajian Mojarrad \\
Institute of Mathematics,
University of Z\"urich\\
Winterthurerstrasse 190, CH 8057 Z\"urich\\
Switzerland}
\date{}
\maketitle
\begin{abstract}
We present a class of arbitrarily high order fully explicit kinetic numerical methods in compressible
fluid dynamics, both in time and space, which include the relaxation
schemes by S. Jin and Z. Xin. These methods can use CFL number larger or equal to unity on regular Cartesian
meshes for multi-dimensional case. These kinetic models depend on a small parameter that can be seen as
a "Knudsen" number. The method is asymptotic preserving in this Knudsen number. Also, the computational
costs of the method are of the same order of a fully explicit scheme. This work is the
extension of Abgrall et al. (2022) \cite{Abgrall} to multi-dimensional systems. 
 We have assessed our method on several  problems for two dimensional scalar problems and Euler equations and the scheme has proven to be robust and to achieve the
theoretically predicted high order of accuracy on smooth solutions.
\end{abstract}
\textbf{keywords.} Kinetic scheme; Compressible fluid dynamics; High order methods; Explicit schemes; Asymptotic preserving; Defect correction method

\section{Introduction}
In this paper, we consider a system of hyperbolic conservation laws in multiple
spatial dimensions
\begin{subequations}\label{eq:1}
\begin{equation}\label{eq:11}
\frac{\partial \bu}{\partial t}+\sum\limits_{i=1}^d \dpar{\bA_i(\bu)}{x_i}  =0,
\end{equation}
with the initial condition
\begin{equation}\label{eq:12}
\bu(\bx,0)=\bu_0(\bx).
\end{equation}
\end{subequations}
where $\bu: \mathbb{R}^d\times \mathbb{R}_{+}\rightarrow \mathbb{R}^K$ and the flux functions $\bA_d$ are locally
Lipschitz continuous on $\mathbb{R}^K$ with values in $\mathbb{R}^K$. We approximate the solution $\bu$ by considering a special class of discrete kinetic systems \cite{Jin,Natalini}. 

Consider a solution $\bbf:\mathbb{R}^d\times \mathbb{R}_{+}\rightarrow \mathbb{R}^{L}$ to the Cauchy problem for the following sequence of semilinear systems
\begin{subequations}\label{eq:2}
\begin{equation}
\label{eq:21}
\frac{\partial \bbf}{\partial t}+\sum\limits_{i=1}^d\Lambda_i \frac{\partial \bbf}{\partial x_i}=\frac{\mathbb{M}(\bu^{\varepsilon})-\bbf}{\varepsilon},
\end{equation}
with the initial condition
\begin{equation}
\label{eq:22}
\bbf(\bx,0)=\bbf_0(\bx).
\end{equation}
\end{subequations}
Here  $\Lambda_i$ are real diagonal $L\times L$ matrices, $\varepsilon$ is a positive number, 
$\mathbb{M}:\mathbb{R}^K\rightarrow \mathbb{R}^L$ is a Lipschitz continuous function, and the function $\bu^{\varepsilon}$ is defined by
\begin{equation*}
 \bu^{\varepsilon}=\sum\limits_{i=1}^N \bbf_i   =\mathbb{P}\bbf
\end{equation*}
where $\mathbb{P}$ is a real constant coefficients $K \times L$ matrix.
To connect problem \eqref{eq:21} - \eqref{eq:22} with problem \eqref{eq:11} - \eqref{eq:12}, we assume that $\mathbb{M}$ is a  Maxwellian function for \eqref{eq:11}, i.e.
\begin{equation}
\label{eq:17}
\begin{cases}
 \mathbb{P}\mathbb{M}(\bu)=\bu,\\
 \mathbb{P}\Lambda_i \mathbb{M}(\bu)=\bA_i(\bu),\quad i=1, \cdots, d.
 \end{cases}
\end{equation}
Clearly, if $\bbf$ converges in some strong topology to a limit $\bg$ and if $\mathbb{P}\bbf_0$ converges to $\bu_0$, then $\mathbb{P}\bg$ is a solution of problem \eqref{eq:11} - \eqref{eq:12}. Actually, the system \eqref{eq:21} is only a BGK approximation for \eqref{eq:11}, see e.g \cite{Bhatnagar,Cercignani} and the references therein. A general stability theory is developed in \cite{Bouchut}, and will implicitly be used throughout this paper, in particular to guaranty that the continuous problem \eqref{eq:2}, equipped with \eqref{eq:17} is well posed.

The method of  \cite{Jin,Aregba}, where the first numerical schemes based on \eqref{eq:2} are described, are based on splitting techniques. As
a result, the order in time is restricted to 2 and can only be improved by nontrivial manipulations \cite{Petra}.
There exists already some ways
for
higher than second order. For example, one approach is to use a relaxed upwind schemes which are proposed running up to CFL 1 and up to third order in
time and space for the finite volume scheme \cite{Schroll}.
A class of high-order weighted essentially nonoscillatory (WENO) reconstructions based on relaxation approximation of hyperbolic systems of conservation laws is presented in \cite{Banda}.
In \cite{Lafitte} a splitting approach is adopted
with a regular CFL stability condition for the overall finite volume scheme. In \cite{Coulette}
a  discontinuous Galerkin method for solving general hyperbolic systems of conservation laws is constructed which is CFL independant and can be
of arbitrary order in time and space. 

Designing algorithms that are
uniformly stable and acccurate in $\varepsilon$  when $\varepsilon \rightarrow 0$ has been an active area of research in recent years. Such schemes are called asymptotic preserving in the sense of Jin \cite{S.Jin}. 
Several asymptotic-preserving methods based on IMEX techniques  have been recently proposed. In \cite{Boscarino,Sebastiano} IMEX Runge–Kutta methods is presented for general hyperbolic systems with relaxation. Specific methods for the Boltzmann equation in the hyperbolic and diffusive
regimes with a computational cost that is independent of $\varepsilon$ is proposed in \cite{Pareschi,Filbet}.

Many of the these methods have inherent limitations with respect to
the order that can be achieved with the time discretization, for example, due to the
time splitting. In \cite{Abgrall} an  arbitrarily high order class of kinetic numerical methods that can run at least at CFL 1 in one dimension is constructed. This work is an extension of \cite{Abgrall} to the multi-dimensional case.

We are interested in a computationally \emph{explicit}
scheme that solves \eqref{eq:2} with uniform accuracy of order $r > 0$ for all $\varepsilon > 0$ and with a
CFL condition, based on the matrices $\Lambda_i,~i=1, \cdots, d$ , that is larger than one for a given system \eqref{eq:1}. We consider the two dimensional case.
The idea is to start from \eqref{eq:21}, we describe first the discretisation of $\Lambda_x \frac{\partial \bbf}{\partial x}$ and $\Lambda_y \frac{\partial \bbf}{\partial y}$. The second step
is to discretise in time. We take into account the source term. The resulting scheme is fully implicit.
The next step is to show that, thanks to the operator $\mathbb{P}$, and using a particular time discretisation, we can
make it computationally explicit, and high order accurate. It is independent of $\varepsilon$.

The paper is organised as follows. In Section \ref{kinetic:sec_time}, we describe the time-stepping algorithm. In Section \ref{kinetic:sec_space}, we describe the space discretisation. In Section \ref{kinetic:sec_stability}, We study the stability of the discretisation of the homogeneous problem. In Section \ref{kinetic:sec_testcases}, we illustrate the robustness and accuracy of the proposed method by means of  numerical tests. Finally, Section \ref{kinetic:sec_conclusion} provide some conclusions and future perspectives.
\section{Time discretisation}
\label{kinetic:sec_time}
Here, we consider the two dimensional case, i.e.
\begin{equation}
\label{eq:13}
\frac{\partial \bbf}{\partial t}+\Lambda_x \frac{\partial \bbf}{\partial x}+\Lambda_y \frac{\partial \bbf}{\partial y}=\frac{\mathbb{M}(\mathbb{P}\bbf)-\bbf}{\varepsilon},
\end{equation}
Knowing the solution $\bbf^{n}$ at time $t_n$, we are looking for the solution at time $t_{n+1}$. First we discretise \eqref{eq:21} in space, we get
\begin{equation}
\label{eq:10}
 \frac{\partial \bbf}{\partial t}+\frac{1}{\Delta x}\Lambda_x \delta ^x\bbf+\frac{1}{\Delta y}\Lambda_y \delta ^y\bbf=\frac{\mathbb{M}(\mathbb{P}\bbf)-\bbf}{\varepsilon},   
\end{equation}
and notice that
\begin{equation}
\label{eq:25}
 \frac{\partial \mathbb{P}\bbf}{\partial t}+\frac{1}{\Delta x}\mathbb{P}(\Lambda_x \delta ^x\bbf)+\frac{1}{\Delta y}\mathbb{P}(\Lambda_y \delta ^y\bbf)=0.  
\end{equation}
Since we want to have a running CFL number of at least one, we use an IMEX defect correction 
method. Using this, we follow \cite{Abgrall} where a defect correction technique can be used and it is made explicit because
the nonlinear term $\mathbb{M}(\mathbb{P}\bbf)$ is explicit.
\subsection{An explicit high order timestepping approach}
The next step is to discretise in time, we subdivide the interval $[t_n,t_{n+1}]$ into sub-intervals obtained from the partition
$$ t_n=t_{(0)}<t_{(1)}<\cdots t_{(p)}<\cdots<t_{(M)}=t_{n+1}$$
with $t_{(p)}=t_n+\beta_p\Delta t$. We  approximate the integral over time using quadrature formula 
\begin{equation*}
\int_{t_n}^{t_n+\beta_p\Delta t} \phi(s)\;ds\approx \Delta t\sum_{q=0}^M w_{pq}\phi(t_n+\beta_q\Delta t)
\end{equation*}
where $w_{pq}$ are the weights. In order to obtain consistent quadrature formula of order $q+1$, we should have
$$w_{pq}=\int_0 ^{\beta_p} l_q(s)\;ds,\quad \sum_{q=0} ^M w_{pq}=\beta_p$$
where $\{l_q\}_{q=0} ^M$ is the Lagrangian basis polynomial of the q-th node.

Let $\bx_l=(x_i,y_j)$ be a fixed grid point, $\bbf_l ^{n,p}\approx\bbf(\bx_l,t_n+\beta_p \Delta t)$ and $\bbf_l ^{n,0}=\bbf_l ^{n}$. We introduce the corrections $r=0, \cdots, R$ for each subinterval $[t_p,t_{p+1}]$ and denote  the solution at the $r$-th correction and the time $t_p$ by $\bbf ^{n,p,r}$.  The notation $\bF$ is the collection of  all the approximations for the sub-steps i.e. the vector $\bF=(\bbf ^{n,1},\cdots,\bbf ^{n,M})^T$.
The notation $\bF^{(r)}$ represents the vector $\bF^{(r)}=(\bbf ^{n,1,r},\cdots,\bbf ^{n,M,r})^T$ i.e. the vector of
all the approximations for the sub-steps at the $r$-th correction.
Now we use a defect correction method and proceed within the time interval $[t_n, t_{n+1}]$ as follows:
\begin{enumerate}
\item For $r=0$, set $\bF^{(0)}=(\bbf ^{n},\cdots,\bbf ^{n})^T$.
\item For each correction $r \geq 0$, define $\bF^{(r+1)}$ by
\begin{equation}
\label{eq:14}
    L_1(\bF^{(r+1)})=L_1(\bF^{(r)})-L_2(\bF^{(r)})
\end{equation}
\item Set $\bF^{n+1}=\bF^{(M)}$.
\end{enumerate}
Formulation \eqref{eq:14} relies on a Lemma which has been proven in \cite{Remi}.

In the following, we introduce the differential operators $L_1$ and $L_2$.
We first define the high order differential operator $L_2$. By integrating \eqref{eq:10} on $[0,t]$, we have
\begin{equation}
\bbf (\bx,t)-\bbf (\bx,0)+\frac{1}{\Delta x} \int\limits_0 ^t \Lambda_x \delta ^x\bbf(\bx,s) \; ds+\frac{1}{\Delta y} \int\limits_0 ^t \Lambda_y \delta ^y\bbf(\bx,s)\;ds=\frac{1}{\varepsilon}\int\limits_0 ^t \Big(\mathbb{M}\big(\mathbb{P}\bbf(\bx,s)\big)-\bbf(\bx,s)\Big)\;ds
\end{equation}
Hence, we get the following approximation for \eqref{eq:13}
\begin{equation}
\label{eq:3}
\bbf_l ^{n,q}-\bbf_l ^{n,0}+\frac{\Delta t}{\Delta x}\big(\sum_{k=0} ^M w_{qk}\Lambda_x\delta_l ^x\bbf^{n,k} \big)+\frac{\Delta t}{\Delta y}\big(\sum_{k=0} ^M w_{qk}\Lambda_y\delta_l ^y\bbf^{n,k} \big)-\frac{\Delta t}{\varepsilon}\sum_{k=0} ^M w_{qk}\big(\mathbb{M}(\mathbb{P} \bbf_l ^{n,k})-\bbf_l ^{n,k}\big)=0
\end{equation}
for $q=1, \cdots,M$. For any $l$, define $[L_2(\bF^{(r)})]_l$
as
\begin{align*}
[L_2(\bF^{(r)})]_l&=\bF_l^{(r)}-\bF_l ^{(0)}+\frac{\Delta t}{\Delta x}\Lambda_xW\delta_l ^x \bF^{(r)}+\frac{\Delta t}{\Delta x}\Lambda_x \bw_0\otimes\delta_l ^x \bbf ^{n,0}+\frac{\Delta t}{\Delta y}\Lambda_yW\delta_l ^y \bF^{(r)}+\frac{\Delta t}{\Delta y}\Lambda_y \bw_0\otimes\delta_l ^y \bbf ^{n,0}\\
&-\frac{\Delta t}{\varepsilon}W\big( \mathbb{M}(\mathbb{P}\bF_l^{(r)})-\bF_l^{(r)}\big)-\frac{\Delta t}{\varepsilon}\bw_0\otimes\big( \mathbb{M}(\mathbb{P}\bbf_l ^{n,0})-\bbf_l ^{n,0}\big)
\end{align*}
where 
$$
W = 
\begin{pmatrix}
	w_{11} & \cdots & w_{1M} \\
	\vdots & \vdots & \vdots \\
	w_{M1} & \cdots & w_{MM} \\	
\end{pmatrix},\quad\bw_0=\begin{pmatrix} 
w_{10}\\ \vdots \\w_{M0}
\end{pmatrix}$$
and
$$\mathbb{M}(\mathbb{P}\bF_l^{(r)})=\big(\mathbb{M}(\mathbb{P}\bbf _l^{n,1,r}),\cdots,\mathbb{M}(\mathbb{P}\bbf_l ^{n,M,r})\big)^T$$
The resulting scheme derived by $L_2$ operator is implicit, and it is very difficult to solve. Now we describe the low order differential operator $L_1$.  We use the forward Euler method on each sub-time step
\begin{equation}
\label{eq:4}
\bbf_l ^{n,q}-\bbf_l ^{n,0}+\beta_q\frac{\Delta t}{\Delta x}\Lambda_x\delta_l ^x\bbf ^{n,0} +\beta_q\frac{\Delta t}{\Delta y}\Lambda_y\delta_l ^y\bbf ^{n,0} -\frac{\Delta t}{\varepsilon}\sum_{k=0} ^M w_{qk}\big(\mathbb{M}(\mathbb{P} \bbf_l ^{n,k})-\bbf_l ^{n,k}\big)=0
\end{equation}
for $q=1 ,\cdots,M$. The low order differential operator $[L_1(\bF^{(r)})]_l$ reads
\begin{align*}
[L_1(\bF^{(r)})]_l&=\bF_l^{(r)}-\bF_l ^{(0)}+\frac{\Delta t}{\Delta x}B\Lambda_x\delta_l ^x \bF^{(0)}+\frac{\Delta t}{\Delta y}B\Lambda_y\delta_l ^y \bF^{(0)}\\
&-\frac{\Delta t}{\varepsilon}W\big( \mathbb{M}(\mathbb{P}\bF^{(r)}_l)-\bF_l^{(r)}\big)-\frac{\Delta t}{\varepsilon}\bw_0\otimes\big( \mathbb{M}(\mathbb{P}\bbf_l ^{n,0})-\bbf_l^{n,0}\big)
\end{align*}
where $B=\text{diag}(\beta_1, \cdots, \beta_M)$. This is still an implicit approximation in time, but the convection part is now explicit. We also note we have kept the same form of the source term approximation in both cases, for reasons that will be explained bellow.

Now we can write \eqref{eq:14} as a multi-step method
where each step writes as
\begin{equation}
\label{eq:15}
\begin{split}
\bF_l^{(r+1)}-\bF_l ^{(0)}&+\frac{\Delta t}{\Delta x}B\Lambda_x\delta_l ^x \bF^{(0)}+\frac{\Delta t}{\Delta y}B\Lambda_y\delta_l ^y \bF^{(0)}
-\frac{\Delta t}{\varepsilon}W\big( \mathbb{M}(\mathbb{P}\bF^{(r+1)}_l)-\bF_l^{(r+1)}\big)
\\&-\frac{\Delta t}{\varepsilon}\bw_0\otimes\big( \mathbb{M}(\mathbb{P}\bbf_l ^{n,0})-\bbf_l^{n,0}\big)
=\frac{\Delta t}{\Delta x}\Lambda_xW(\delta_l ^x \bF^{(0)}-\delta_l ^x \bF^{(r)})+\frac{\Delta t}{\Delta y}\Lambda_yW(\delta_l ^y \bF^{(0)}-\delta_l ^y \bF^{(r)})
\end{split}
\end{equation}
for any $l$. By applying $\mathbb{P}$ to this equation, we will obtain the following equation for calculating $\mathbb{P}\bF_l^{(r+1)}$
\begin{equation}
\label{eq:16}
\mathbb{P}\bF_l^{(r+1)}=\mathbb{P}\bF_l^{(0)}-\frac{\Delta t}{\Delta x}\mathbb{P}\Lambda_xW\delta_l ^x \bF^{(r)}-\frac{\Delta t}{\Delta x}\bw_0\otimes\mathbb{P} \Lambda_x\delta_l ^x \bbf^{n,0}-\frac{\Delta t}{\Delta y}\mathbb{P}\Lambda_yW\delta_l ^y \bF^{(r)}-\frac{\Delta t}{\Delta y}\bw_0\otimes\mathbb{P} \Lambda_y\delta_l ^y \bbf^{n,0}
\end{equation}
and we substitute $\mathbb{P}\bF_l^{(r+1)}$ into the Maxwellian in the \eqref{eq:15}. Alternatively, one can rewrite \eqref{eq:15} as follows
\begin{equation}
\label{eq:7}
\begin{split}
(\text{Id}_{M\times M}+\frac{\Delta t}{\varepsilon}W)\bF_l^{(r+1)}&=\frac{\Delta t}{\varepsilon}W\mathbb{M}(\mathbb{P}\bF_l^{(r+1)}) +\bF_l^{(0)}-\frac{\Delta t}{\Delta x}\Lambda_xW\delta_l ^x \bF^{(r)}
-\frac{\Delta t}{\Delta x}\bw_0\otimes \Lambda_x\delta_l ^x \bbf^{n,0}\\
&-\frac{\Delta t}{\Delta y}\Lambda_yW\delta_l ^y \bF^{(r)}
-\frac{\Delta t}{\Delta y}\bw_0\otimes \Lambda_y\delta_l ^y \bbf^{n,0}+\frac{\Delta t}{\varepsilon}\bw_0\otimes\big( \mathbb{M}(\mathbb{P}\bbf_l ^{n,0})-\bbf_l^{n,0}\big)
\end{split}
\end{equation}
If $\text{Id}_{M\times M}+{\Delta t}W/{\varepsilon}$ is invertible, the defect correction computes the solution at time $t_{n+1}$ using $M$ steps of the form
\begin{subequations}
\label{eq8}
\begin{equation}
\label{eq:8}
\begin{split}
\bF_l^{(r+1)}&=(\text{Id}_{M\times M}+\frac{\Delta t}{\varepsilon}W)^{-1}\bigg( \frac{\Delta t}{\varepsilon}W\mathbb{M}(\mathbb{P}\bF_l^{(r+1)})+\bF_l^{(0)}-\frac{\Delta t}{\Delta x}\Lambda_xW\delta_l ^x \bF^{(r)}
-\frac{\Delta t}{\Delta x}\bw_0\otimes \Lambda_x\delta_l ^x \bbf^{n,0}\\
&-\frac{\Delta t}{\Delta y}\Lambda_yW\delta_l ^y \bF^{(r)}
-\frac{\Delta t}{\Delta y}\bw_0\otimes \Lambda_y\delta_l ^y \bbf^{n,0}+\frac{\Delta t}{\varepsilon}\bw_0\otimes\big( \mathbb{M}(\mathbb{P}\bbf_l ^{n,0})-\bbf_l^{n,0}\big)
\bigg).
\end{split}
\end{equation}
For computing $\mathbb{P}\bF_l ^{(r+1)}$, for any $l$, we rewrite \eqref{eq:16} as (for simplicity, we drop the superscript $n$)
\begin{equation}\label{eq:8:2}
\mathbb{P}\bbf_l ^{q,r+1}-\mathbb{P}\bbf_l ^{0}+\frac{\Delta t}{\Delta x}\sum_{k=0} ^M w_{qk}\mathbb{P}\Lambda_x\delta_l ^x\bbf^{k,r} +\frac{\Delta t}{\Delta y}\sum_{k=0} ^M w_{qk}\mathbb{P}\Lambda_y\delta_l ^y\bbf^{k,r}=0,
\end{equation}
\end{subequations}
for $q=1, \cdots, M$. 

We  write the increment $\delta_l \bbf^{k}$ as a sum of residuals, as follows
\begin{subequations}
    \label{residual}
\begin{equation}\label{residual:1}
\begin{split}
\delta_l \bbf^k&=\Delta y\Lambda_x(\hat{\bbf}_{i+\frac{1}{2},j}^k-\hat{\bbf}_{i-\frac{1}{2},j}^k) +\Delta x\Lambda_y(\hat{\bbf}_{i,j+\frac{1}{2}}^k-\hat{\bbf}_{i,j-\frac{1}{2}}^k)\\
&=\Phi_{l,(i,j)} ^{[i,i+1]\times [j,j+1],k}+
\Phi_{l,(i,j)} ^{[i-1,i]\times [j,j+1],k}+
\Phi_{l,(i,j)} ^{[i,i+1]\times [j-1,j],k}+
\Phi_{l,(i,j)} ^{[i-1,i]\times [j-1,j],k}
\end{split}
\end{equation}
with
\begin{equation}\label{residual:2}
\begin{split}
\Phi_{l,(i,j)} ^{[i,i+1]\times [j,j+1],k}&=\frac{1}{2}\bigg (\Lambda_x
\big ( \hbbf_{i+1/2,j}-\bbf_{ij})\Delta y+ \Lambda_y\big (\hbbf_{i,j+1/2}-\bbf_{ij})\big )\Delta y\bigg ),  \\
\Phi_{l,(i,j)} ^{[i-1,i]\times [j,j+1],k}&=\frac{1}{2}\bigg (\Lambda_x
\big (\hbbf_{i+1/2,j}-\bbf_{ij}\big )\Delta y+ \Lambda_y(\bbf_{ij}-\hbbf_{i,j-1/2}\big )\Delta x\bigg ),
\\
\Phi_{l,(i,j)} ^{[i-1,i]\times [j-1,j],k}&= \frac{1}{2}\bigg ( \Lambda_x\big (\bbf_{ij}-\hbbf_{i-1/2,j}\big )\Delta y+\Lambda_y\big ( \bbf_{ij}-\hbbf_{ij-1/2}\big )\Delta x \bigg ), \\
\Phi_{l,(i,j)} ^{[i-1,i]\times [j-1,j],k}&=\frac{1}{2}\bigg ( \Lambda_x\big (\bbf_{ij}-\hbbf_{i-1/2,j}\big )\Delta y+\Lambda_y\big (  \hbbf_{i,j+1/2}-\bbf_{ij}\big ) \Delta x\bigg ).  \\
\end{split}
\end{equation}
\end{subequations}
We see that 
\begin{equation}\label{residual:conservation}
\begin{split}
\Phi_{l,(i,j)} ^{[i,i+1]\times [j,j+1],k}&+\Phi_{l,(i+1,j)} ^{[i,i+1]\times [j,j+1],k}+\Phi_{l,(i+1,j+1)} ^{[i,i+1]\times [j,j+1],k}+\Phi_{l,(i,j+1)} ^{[i,i+1]\times [j,j+1],k}\\
&=
\frac{\Delta x}{2}\Lambda_x\big ( \bbf_{i+1,j}+\bbf_{i+1,j+1}\big )-\frac{\Delta x}{2}\Lambda_x\big (\bbf_{ij}+\bbf_{i,j+1}\big ) \\
& +\frac{\Delta y}{2}\Lambda_y\big ( \bbf_{i+1,j}+\bbf_{i+1,j+1}\big )-\frac{\Delta y}{2}\Lambda_y\big ( \bbf_{ij}+\bbf_{i,j+1}\big )
,\end{split}
\end{equation}
\begin{figure}[H]
\centering
{\includegraphics[width=0.35\textwidth]{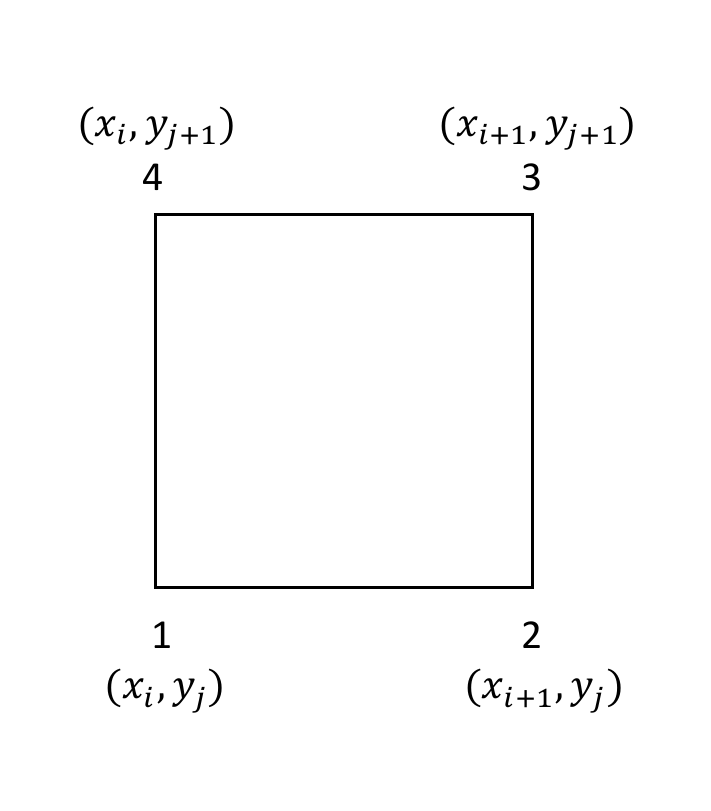}
}
\caption{Sample quad element.}
\label{fig:element}
\end{figure}
i.e. the sum of the residuals assigned to each of the four vertices of  the quad $[x_{i},x_{i+1}]\times [y_j,y_{j+1}]$ is equal to an approximation of the flux across the boundary of this quad, see Fig. \ref{fig:element}.
This guaranties local conservation, see \cite{AbgrallRD}.
The reason of writing the space update as as sum of residual will become clear in section \ref{sec:limit}.

Now we explain  the residuals for an example where the  velocities are orthogonal. This case is suggested in \cite{Aregba}. In order to construct the
system \eqref{eq:21} one must find $\mathbb{P}$, $\mathbb{M}$, and $\Lambda_j ~(j=1,\cdots,d)$ such that the consistency relations \eqref{eq:17} are satisfied.
We take $L=N \times K$, $\mathbb{P}=(\text{Id}_{K\times K}, \cdots,\text{Id}_{K\times K})$ with $N$ blocks $\text{Id}_{K\times K}$, the identity matrix in $\mathbb{R}^K$. 
Each matrix $\Lambda_j~ (j=1,\cdots,d)$ is constituted of $N$ diagonal blocks of
size $K\times K$, i.e.
\begin{equation*}
 \Lambda_j=\text{diag}(C_1^{(j)},\cdots,C_N^{(j)}),\quad   C_i^{(j)}=\lambda_i^{(j)}\text{Id}_{K\times K},\quad \lambda_i^{(j)}\in \mathbb{R}.
\end{equation*}
where $j=1,\cdots,N$. Then \eqref{eq:21} can be rewritten as
\begin{equation}
\frac{\partial \bbf_j}{\partial t}+\sum\limits_{i=1}^d\lambda_j^{(i)} \frac{\partial \bbf_j}{\partial x_i}=\frac{\mathbb{M}_j(\bu^{\varepsilon})-\bbf_j}{\varepsilon},\, j=1,\cdots,N.
\end{equation}
in which $\bu^{\varepsilon}=\sum\limits_{j=1}^N \bbf_n$. This example works with any number of blocks provided that $N\geq d+1$. We take the velocities such that 
\begin{equation*}
  \Lambda_j=\text{diag}(\lambda_{ij}\text{Id} _{K\times K})_{1 \leq i \leq N},
  \end{equation*}
\begin{equation*}
  \sum_{i=1} ^{N} \lambda_{ij}=0,~ j=1,\cdots , d,
\end{equation*}  
and
\begin{equation*}
  \sum_{i=1} ^{N} \lambda_{il}\lambda_{ij}=0,~ j,l=1,\cdots d,~ j\neq l.
\end{equation*}  
Therefore, the 
Maxwellian function is given by
\begin{equation*}
\mathbb{M}_{i}(\bu)=\frac{\bu}{N}+\sum_{j=1} ^d \frac{\bA_j(\bu)}{a_j ^2} \lambda_{ij},~i=1,\cdots,N.
\end{equation*}
where $a_j ^2=\sum_{i=1} ^{N} \lambda_{ij}^2$.
We fix $J,N'\geq1$, $\lambda>0$ and set
\begin{equation*}
\begin{cases}
\lambda_j=\frac{n}{J}\lambda &1 \leq j\leq J,~J\geq 1,\\
  \Lambda_x=\text{diag}\Big(\lambda_n\text{diag}\big( \cos(\frac{i\pi}{2N^{'}})\text{ I} \big)_{1\leq i\leq 4N^{'}}\Big) &N^{'}\geq 1,
 \\
  \Lambda_y=\text{diag}\Big(\lambda_n\text{diag}\big( \sin(\frac{i\pi}{2N^{'}})\text{ I} \big)_{1\leq i\leq 4N^{'}}\Big).
  \end{cases}
\end{equation*}
Here we have $N=4N'J$. We define the Maxwellian functions  by
\begin{equation*}
\mathbb{M}_{j}(\bu)=\frac{1}{N}\Big( \bu+\frac{12\;i\;J}{\lambda(J+1)(2J+1)}(\bA_1\cos\frac{j\pi}{2N^{'}}+\bA_2\sin\frac{j\pi}{2N^{'}}) \Big), ~1\leq j\leq 4N^{'}, 1\leq i\leq J.
\end{equation*}
Now, we consider the special case where $J=N'=1$ and  $N=4$. Set
\begin{equation*}
\begin{cases}
  \Lambda_x=\text{diag}\Big(\lambda\text{ diag}\big( \cos(\frac{i\pi}{2})\text{ I} \big)_{1\leq i\leq 4}\Big),
  \\
  \Lambda_y=\text{diag}\Big(\lambda\text{ diag}\big( \sin(\frac{i\pi}{2})\text{ I} \big)_{1\leq i\leq 4}\Big).
  \end{cases}
\end{equation*}
Therefore the Maxwellian functions  are
\begin{equation*}
\mathbb{M}_{i}(\bu)=\frac{1}{4}\Big( \bu+\frac{2}{\lambda}(\bA_1\cos\frac{i\pi}{2}+\bA_2\sin\frac{i\pi}{2}) \Big).
\end{equation*}
for $1\leq i \leq4$.

\subsection{Error estimate}
The natural question is: how many iterations should we do in the method \eqref{eq8} to recover the time and space accuracy?

Again, we will consider the two dimensional version of \eqref{eq:11}. Let us consider $\varphi: \mathbb{R}^2\rightarrow \mathbb{R}^K$, which is continuously differentiable on $\mathbb{R}^2$ with values in $\mathbb{R}^K$ and has a compact support. We will consider the discrete version of $L^2$ and $H^1$ norms of $\varphi$ as follows
\begin{equation*}
\Vert \varphi \Vert _{L^2}^2 =  \sum\limits_{i,j\in \mathbb{Z}}\Delta x\Delta y \Vert \varphi_{i,j} \Vert^2,\quad \Vert \varphi \Vert _{H^1}^2 =\Vert \varphi \Vert _{L^2}^2+\sum\limits_{i,j\in \mathbb{Z}}\Delta x\Delta y  \Big(\Big \Vert \dfrac{\varphi_{i+1,j}-\varphi_{ij}}{\Delta x}
\Big \Vert^2+\Big \Vert\dfrac{\varphi_{i,j+1}-\varphi_{i,j}}{\Delta y}
\Big \Vert^2\Big).
\end{equation*}
In the following, we will set:
$$D_{i+1/2,j}\varphi=\dfrac{\varphi_{i+1,j}-\varphi_{ij}}{\Delta x},
D_{i,j+1/2}\varphi=\dfrac{\varphi_{i,j+1}-\varphi_{i,j}}{\Delta y}\text{ and } D\varphi_{ij}=(D_{i+1/2,j}\varphi,D_{i,j+1/2}\varphi).$$ We have a Poincar\'e inegality, on any compact. 

One can consider the discrete equivalent of $L^2_{loc}$ and $H^{-1}_{loc}$ in an interval $I=[a,b]\times [c,d]$ as follows
\begin{equation*}
 \Vert  \bF \Vert_{2,I}=\sup\limits_{\varphi \in C_0^1(I)^K}  \frac{\sum\limits_{i,j\in \mathbb{Z}}\Delta x\Delta y \langle \varphi_{i,j},\bbf_{i,j} \rangle}{\Vert \varphi \Vert _{L^2}},\quad  \Vert  \bF \Vert_{-1,I}=\sup\limits_{\varphi \in C_0^1(I)^K}  \frac{\sum\limits_{i,j\in \mathbb{Z}}\Delta x\Delta y \langle \varphi_{i,j},\bbf_{i,j} \rangle}{\Vert \varphi \Vert _{H^1}}.
\end{equation*}
We have a first result on the $H^{-1}_{loc}$ estimate of $L_2(\bF)-L_1(\bF)$.
\begin{lemma}
If $\widehat{\bF}_{i+\frac{1}{2},j}=\sum\limits_{l=-r}^s\alpha_l \bF_{i+l,j}$ and $\widehat{\bF}_{i,j+\frac{1}{2}}=\sum\limits_{l=-r'}^{s'}\alpha'_l \bF_{i,j+l}$, we have
\begin{equation*}
\Vert  L_2(\bF)-L_1(\bF)\Vert _{-1,I}\leq \gamma\Delta t\Vert\bF \Vert _{2,I}  
\end{equation*}
 where $\gamma=\max\{\max\limits_{-r\leq l\leq s} |\alpha_l|,\max\limits_{-r'\leq l\leq s'}|\alpha'_l|\}\times \max\{\max\limits_m|\lambda_m^x|,\max\limits_n|\lambda_n^y| \}$.
 \label{lemma:1}
\end{lemma}
\begin{proof}
\begin{align*}
    \Big|\sum\limits_{i,j}\Delta x\Delta y \langle &\varphi_{i,j},[L_2(\bF)]_{i,j}-[L_1(\bF)]_{i,j} \rangle\Big|^2\\&=\Big|\sum\limits_{i,j} \Delta t\Delta y \langle \varphi_{i,j},\Lambda_xW(\widehat{\bF}_{i+\frac{1}{2},j}-\widehat{\bF}_{i-\frac{1}{2},j}) \rangle+\sum\limits_{i,j} \Delta t\Delta x \langle \varphi_{i,j},\Lambda_yW(\widehat{\bF}_{i,j+\frac{1}{2}}-\widehat{\bF}_{i,j-\frac{1}{2}}) \rangle \Big|^2\\
    &=\Big|\sum\limits_{i,j} \Delta t\Delta x\Delta y \langle D_{i+\frac{1}{2},j}\varphi,\Lambda_xW\widehat{\bF}_{i+\frac{1}{2},j} \rangle+\sum\limits_{i,j} \Delta t\Delta x\Delta y \langle D_{i,j+\frac{1}{2}}\varphi,\Lambda_yW\widehat{\bF}_{i,j+\frac{1}{2}} \rangle \Big|^2\\
    &\leq\Delta t^2\Vert \Lambda_x \Vert^2 \Vert W \Vert^2\Big(\sum\limits_{i,j} \Delta x\Delta y \Vert D_{i+\frac{1}{2},j}\Vert^2+\sum\limits_{i,j}\Delta x \Delta y\Vert D_{i,j+\frac{1}{2}}\varphi \Vert^2\Big)\\&\times\Big(\sum\limits_{i,j} \Delta x\Delta y \Vert \widehat{\bF}_{i+\frac{1}{2},j}\Vert^2
    +\sum\limits_{i,j}\Delta x \Delta y\Vert \widehat{\bF}_{i,j+\frac{1}{2}}\Vert^2 \Big) \\ &\leq  \gamma^2\Delta t^2\Vert \varphi \Vert_{H^1}^2\Vert \bF\Vert_{2,I} ^2.
\end{align*}
Then the proof is complete.
\end{proof}
Before we proceed to proposition \ref{prop:1}, we need a further result on the behaviour of $L_1$.
\begin{lemma}
\label{lemma:2}
Let us consider $\bF,\bF'$ and $\bG,\bG'$ such that
$$
[L_1(\bF)]_l=\bG_l,\quad [L_1(\bF')]_l=\bG'_l
$$
for some $l$. And we assume that 
there
exists $\gamma,\gamma'>0$ such that
\begin{equation*}
\Vert  (\text{Id}_{M\times M}+\frac{\Delta t}{\varepsilon}W)^{-1} \Vert  \leq \gamma,\quad  \frac{\Delta t}{\varepsilon}\Vert  (\text{Id}_{M\times M}+\frac{\Delta t}{\varepsilon}W)^{-1}W \Vert  \leq \gamma', 
\end{equation*}
then there exists constant $\eta>0$, independent of $\bF, \bF', I$ and $\varepsilon$ such that
\begin{align*}
 \Vert  \bF_l- \bF'_l \Vert_{2,I}\leq \eta \Vert  \bG_l- \bG'_l \Vert_{2,I}, \quad  \Vert  \bF_l- \bF'_l \Vert_{-1,I}\leq \eta \Vert  \bG_l- \bG'_l \Vert_{-1,I},
\end{align*}
\end{lemma}
\begin{proof}
To prove the lemma, we apply $\mathbb{P}$ to $[L_1(\bF)]_l=\bG_l$, we get
\begin{equation*}
 \mathbb{P}\bF_l=\mathbb{P}\bG_l+\mathbb{P}\bF_l ^{(0)}-\frac{\Delta t}{\Delta x}\mathbb{P}B\Lambda_x\delta_l ^x \bF^{(0)}-\frac{\Delta t}{\Delta y}\mathbb{P}B\Lambda_y\delta_l ^y \bF^{(0)}
\end{equation*}
Now, we substitute $\mathbb{P}\bF$ into the Maxwellian in the equation $[L_1(\bF)]_l=\bG_l$, we have
\begin{align*}
 \bF_l&=\bG_l +\bF_l ^{(0)}-\frac{\Delta t}{\Delta x}B\Lambda_x\delta_l ^x \bF^{(0)}-\frac{\Delta t}{\Delta y}B\Lambda_y\delta_l ^y \bF^{(0)}\\
&+\frac{\Delta t}{\varepsilon}W\Big[ \mathbb{M}\big (\mathbb{P}\bG_l+\mathbb{P}\bF_l ^{(0)}-\frac{\Delta t}{\Delta x}\mathbb{P}B\Lambda_x\delta_l ^x \bF^{(0)}-\frac{\Delta t}{\Delta y}\mathbb{P}B\Lambda_y\delta_l ^y \bF^{(0)}\big )-\bF_l\Big]+\frac{\Delta t}{\varepsilon}\bw_0\otimes\Big( \mathbb{M}(\mathbb{P}\bbf_l ^{n,0})-\bbf_l^{n,0}\Big)  
\end{align*}
We can
collect all the unknown terms $\bF_l$ on the left hand side
\begin{align*}
 \bF_l&=(\text{Id}_{M\times M}+\frac{\Delta t}{\varepsilon}W)^{-1} \Big[\bG_l +\bF_l ^{(0)}-\frac{\Delta t}{\Delta x}B\Lambda_x\delta_l ^x \bF^{(0)}-\frac{\Delta t}{\Delta y}B\Lambda_y\delta_l ^y \bF^{(0)}\\
&+\frac{\Delta t}{\varepsilon}W \mathbb{M}\big (\mathbb{P}\bG_l+\mathbb{P}\bF_l ^{(0)}-\frac{\Delta t}{\Delta x}\mathbb{P}B\Lambda_x\delta_l ^x \bF^{(0)}-\frac{\Delta t}{\Delta y}\mathbb{P}B\Lambda_y\delta_l ^y \bF^{(0)}\big )+\frac{\Delta t}{\varepsilon}\bw_0\otimes\big( \mathbb{M}(\mathbb{P}\bbf_l ^{n,0})-\bbf_l^{n,0}\big)  \Big ]
\end{align*}
We have
\begin{equation*}
  \bF_l-\bF'_l= (\text{Id}_{M\times M}+\frac{\Delta t}{\varepsilon}W)^{-1} \Big(\bG_l-\bG'_l+\frac{\Delta t}{\varepsilon}W\mathbb{M}(\mathbb{P}\bG_l-\mathbb{P}\bG'_l)\Big).
\end{equation*}
This concludes the proof.
\end{proof}
\begin{proposition}
\label{prop:1}
Under the conditions of lemmas \ref{lemma:1} and \ref{lemma:2}, if $\bF^{*}$ is the unique
solution of $L_2(\bF)=0$, there exists $\theta$ independent of $\varepsilon$ such that we have
\begin{equation*}
\Vert  \bF^{(r+1)}- \bF^{*} \Vert_{L^2}\leq  (\theta \Delta t)^{r+1} \Vert  \bF^{(0)}- \bF^{*} \Vert_{L^2}.
\end{equation*}
\end{proposition}
\begin{proof}
We have
\begin{align*}
  L_1(\bF^{(r+1)})- L_1(\bF^{*})&=L_1(\bF^{(r)})-L_2(\bF^{(r)})- L_1(\bF^{*})\\
  &=\big(L_1(\bF^{(r)})- L_1(\bF^{*})\big)-\big( L_2(\bF^{(r)})- L_2(\bF^{*}) \big)
\end{align*}
Also, we can write a Poincar\'e  inequality for $\varphi$ as follows
\begin{equation*}
    \Vert \varphi \Vert_{2,I}\leq C\Vert D\varphi \Vert_{2,I}.
\end{equation*}
where $C$ is constant. Using results from Poincar\'e  inequality, lemmas \ref{lemma:1} and \ref{lemma:2}, the proof is concluded.
\end{proof}
Starting from the Chapman-Enskog expansion of the \eqref{eq:21}, we can show that our method is
asymptotic preserving. We have
\begin{equation}
\label{eq:9}
\begin{split}
 \bbf&=\mathbb{M}(\bu^{\varepsilon})+\mathcal{O}(\varepsilon), \\
 \dpar{ \bu^{\varepsilon}}{ t}&+\dpar{\bA_1(\bu^{\varepsilon})}{x}+\dpar{\bA_2(\bu^{\varepsilon})}{y}=\mathcal{O}(\varepsilon).
\end{split}
\end{equation}
\begin{proposition}
\eqref{eq:8} is consistent with the limit model \eqref{eq:9} up to an $\mathcal{O}(\varepsilon)$.
\end{proposition}
\begin{proof}
Let us define $\bu^{\varepsilon,(r)}=\mathbb{P}\bF^{(r)}$, from \eqref{eq:16} we get
\begin{equation*}
 \bu^{\varepsilon,(r+1)}=\mathbb{P}\bF^{(0)}-\frac{\Delta t}{\Delta x}\mathbb{P}\Lambda_xW\delta ^x \bF^{(r)}-\frac{\Delta t}{\Delta x}\bw_0\otimes\mathbb{P} \Lambda_x\delta ^x \bbf^{n,0}-\frac{\Delta t}{\Delta y}\mathbb{P}\Lambda_yW\delta ^y \bF^{(r)}-\frac{\Delta t}{\Delta y}\bw_0\otimes\mathbb{P} \Lambda_y\delta ^y \bbf^{n,0}+\mathcal{O}(\varepsilon)   
\end{equation*}
Using \eqref{eq:17}, we have
\begin{align*}
    \bF^{(r+1)}&=\mathbb{M}\big(\mathbb{P}\bF^{(0)}-\frac{\Delta t}{\Delta x}\mathbb{P}\Lambda_xW\delta ^x \bF^{(r)}-\frac{\Delta t}{\Delta x}\bw_0\otimes\mathbb{P} \Lambda_x\delta ^x \bbf^{n,0}\\&-\frac{\Delta t}{\Delta y}\mathbb{P}\Lambda_yW\delta ^y \bF^{(r)}-\frac{\Delta t}{\Delta y}\bw_0\otimes\mathbb{P} \Lambda_y\delta ^y \bbf^{n,0}\big)+\mathcal{O}(\varepsilon)\\
    &=\mathbb{M}(\bu^{\varepsilon,(r+1)})+\mathcal{O}(\varepsilon)
\end{align*}
Then
\begin{align*}
 \bu^{\varepsilon,(r+1)}&=\bu^{\varepsilon,(0)} -\frac{\Delta t}{\Delta x}W\delta ^x\mathbb{P}\Lambda_x \bF^{(r)}-\frac{\Delta t}{\Delta x}\bw_0\otimes\delta^{x}\mathbb{P}\Lambda_x   \bbf^{n,0}-\frac{\Delta t}{\Delta y}W\delta ^y\mathbb{P}\Lambda_y \bF^{(r)}-\frac{\Delta t}{\Delta y}\bw_0\otimes\delta ^y\mathbb{P} \Lambda_y \bbf^{n,0}+\mathcal{O}(\varepsilon)  \\
 &=\bu^{\varepsilon,(0)} -\frac{\Delta t}{\Delta x}W\delta ^x\bA_1(\bu^{\varepsilon,(r)})-\frac{\Delta t}{\Delta x}\bw_0\otimes\delta^{x}\bA_1(\bu^{\varepsilon,(0)})-\frac{\Delta t}{\Delta y}W\delta ^y\bA_2(\bu^{\varepsilon,(r)})\\&-\frac{\Delta t}{\Delta y}\bw_0\otimes\delta^{y}\bA_2(\bu^{\varepsilon,(0)})+\mathcal{O}(\varepsilon).
\end{align*}
We observe that if the the spatial discretisation is consistent with the  derivative in space, the above formula is the space discretisation  and time of the asymptotic model \eqref{eq:9}, and the result is concluded.
\end{proof}
\subsection{Time discretisation in the \texorpdfstring{$L_2$}{Lg} operator}
We consider first, second and fourth order approximation in time in the $L_2$
operator, when there is no source term. 
\begin{enumerate}
    \item 
For the first order approximation, we get
\begin{equation*}
  \mathbb{P}\bbf^{n,1}-\mathbb{P}\bbf^{n,0}+\frac{\Delta t}{\Delta x}\mathbb{P}(\Lambda _x\delta^x\bbf^{n,1})+\frac{\Delta t}{\Delta y}\Lambda _y\mathbb{P}(\delta^y\bbf^{n,1})=0,  
\end{equation*}
where $\bbf^{n,0}=\bbf^n$ and $\bbf^{n,1}\approx \bbf(t_{n+1})$. The matrix $W$ becomes $W=(1)$. Then, we get
\begin{equation*}
    (\text{Id}_{1\times 1}+\frac{\Delta t}{\varepsilon}W)^{-1}=\frac{\varepsilon}{\varepsilon+\Delta t},\quad \frac{\Delta t}{\varepsilon}(\text{Id}_{1\times 1}+\frac{\Delta t}{\varepsilon}W)^{-1}W=\frac{\Delta t}{\varepsilon+\Delta t},
\end{equation*}
We observe that these two matrices are uniformly bounded.
\item
For the second order approximation, which is Crank-Nicholson method, the scheme becomes
\begin{equation*}
   \mathbb{P} \bbf^{n,1}-\mathbb{P}\bbf^{n,0}+\frac{\Delta t}{\Delta x}\big(\frac{1}{2}\mathbb{P}(\Lambda _x\delta^x\bbf^{n,0})+\frac{1}{2}\mathbb{P}(\Lambda _x\delta^x\bbf^{n,1})\big)+\frac{\Delta t}{\Delta y}\big(\frac{1}{2}\mathbb{P}(\Lambda _y\delta^y\bbf^{n,0})+\frac{1}{2}\mathbb{P}(\Lambda _y\delta^y\bbf^{n,1})\big)=0,
\end{equation*}
also we have $W=(\frac{1}{2})$.
Similarly, we see that two matrices $(\text{Id}_{1\times 1}+\frac{\Delta t}{\varepsilon}W)^{-1}$ and $\frac{\Delta t}{\varepsilon}(\text{Id}_{1\times 1}+\frac{\Delta t}{\varepsilon}W)^{-1}W$ are uniformly bounded.
\item
For the fourth order scheme \cite{Hairer}, we get 
\begin{align*}
    \mathbb{P}\bbf^{n,1}-\mathbb{P}\bbf^{n,0}&+\frac{\Delta t}{\Delta x}\big(\frac{5}{24}\mathbb{P}(\Lambda _x\delta^x\bbf^{n,0})+\frac{1}{3}\mathbb{P}(\Lambda _x\delta^x\bbf^{n,1})-\frac{1}{24}\mathbb{P}(\Lambda _x\delta^x\bbf^{n,2})\big)\\&+\frac{\Delta t}{\Delta y}\big(\frac{5}{24}\mathbb{P}(\Lambda _y\delta^y\bbf^{n,0})+\frac{1}{3}\mathbb{P}(\Lambda _y\delta^y\bbf^{n,1})-\frac{1}{24}\mathbb{P}(\Lambda _y\delta^y\bbf^{n,2})\big)=0,
\\
   \mathbb{P} \bbf^{n,2}-\mathbb{P}\bbf^{n,0}&+\frac{\Delta t}{\Delta x}\big(\frac{1}{6}\mathbb{P}(\Lambda _x\delta^x\bbf^{n,0})+\frac{2}{3}\mathbb{P}(\Lambda _x\delta^x\bbf^{n,1})+\frac{1}{6}\mathbb{P}(\Lambda _x\delta^x\bbf^{n,2})\big)\\&+\frac{\Delta t}{\Delta y}\big(\frac{1}{6}\mathbb{P}(\Lambda _y\delta^y\bbf^{n,0})+\frac{2}{3}\mathbb{P}(\Lambda _y\delta^y\bbf^{n,1})+\frac{1}{6}\mathbb{P}(\Lambda _y\delta^y\bbf^{n,2})\big)=0,
\end{align*}
where $\bbf^{n,0}=\bbf^n$, $\bbf^{n,1}\approx\bbf(t_n+\frac{\Delta t}{2})$ and $\bbf^{n,2}\approx \bbf(t_{n+1})$. Also, we have $$W=\begin{pmatrix}
	\frac{1}{3}&  -\frac{1}{24} \\
	\frac{2}{3}  & \frac{1}{6} 	
\end{pmatrix}$$
It is easy to observe that the matrix $\text{Id}_{2\times 2}+\frac{\Delta t}{\varepsilon}W$ is invertible and the matrices $(\text{Id}_{2\times 2}+\frac{\Delta t}{\varepsilon}W)^{-1}$ and $\frac{\Delta t}{\varepsilon}(\text{Id}_{2\times 2}+\frac{\Delta t}{\varepsilon}W)^{-1}W$
are uniformly bounded.
\end{enumerate}
\section{Space discretisation}
\label{kinetic:sec_space}
\subsection{Definition of the \texorpdfstring{$\delta^x$}{Lg} and \texorpdfstring{$\delta^y$}{Lg} operators}
Since $\Lambda_x$ and $\Lambda_y$ are diagonal matrices, we can consider the scalar transport equation
\begin{equation*}
    \dpar{f}{t}+a\dpar{f}{x}+b\dpar{f}{y}=0,
\end{equation*}
where $a$ and $b$ are constants and both of them can not be zero at the same time. Next, we will discuss the approximation of $\dpar{f_l}{x}$, the approximation for $\dpar{f_l}{y}$
is obtained in a similar manner. In \cite{Iserles} has been developed a stable numerical finite difference methods for first-order hyperbolics in one dimension,
which use $s$ forward and $r$ backward steps in the discretisation of the space derivatives that are of order at most $2\min \{r+1,s\}$.  These methods are based on the approximation
of $\dpar{f}{x}(x_i,y_j)$, $i,j \in \mathbb{Z}$, by a finite difference
\begin{equation*}
 \frac{1}{\Delta x} \sum\limits_{k=-r}^{s}\alpha_kf_{i+k,j}  
\end{equation*}
when $a>0$ and we say that the method is of the class $\{r,s\}$. If $a<0$, we set
\begin{equation*}
 \frac{1}{\Delta x} \sum\limits_{k=-s}^{r}\alpha_{-k}f_{i+k,j}  
\end{equation*}
Throughout this section we suppose without loss of generality that $a > 0$. Otherwise the roles of $r$ and $s$ are reversed. We call an $\{r, s\}$ method of the highest order an interpolatory method, and is denoted
by $[r.s]$. Following the theorem of \cite{Iserles}, for all integers $r,s\geq 0$ the order of the interpolatory method $[r.s]$ is $q=r+s$, i.e.
\begin{equation*}
 \frac{\delta^x f_{i,j}}{\Delta x}-\dpar{f}{x}(x_i,y_j) =c_{r,s}(\Delta x)^q\frac{\partial^{q+1}f}{\partial x^{q+1}}(x_i,y_j)+\mathcal{O} ((\Delta x)^{q+1}) 
\end{equation*}
where the error constant $c_{r,s}$ is defined by
\begin{equation*}
 c_{r,s}=\frac{(-1)^{s-1}r!s!}{(r+s+1)!}  
\end{equation*}
The coefficients are defined by
\begin{align*}
 \alpha_k&=\frac{(-1)^{k+1}}{k} \frac{r!s!}{(r+k)!(s-k)!},\quad -r\leq k\leq s,\quad k\neq0,\\
 \alpha_0&=-\sum\limits_{k=-r,k\neq 0}^s \alpha_k,
\end{align*}
We recall that the $[r, r]$, $[r, r+1]$ and $[r, r+2]$ schemes are the only
stable interpolatory methods, and we will only consider these approximations. Before
we proceed to the possible choices for $\delta$, it is important to remark that, we can  write
\begin{equation*}
  \delta^x f_{i,j}=  \hat{f}_{i+\frac{1}{2},j}-\hat{f}_{i-\frac{1}{2},j}
\end{equation*}
where
\begin{equation*}
 \hat{f}_{i+\frac{1}{2},j}=\sum\limits_{k=-s+1}^{r} \beta_k f_{i+k,j}  
\end{equation*}
with $\beta_k=\sum\limits_{m\geq k} \alpha_m$.
In the following, we consider first, second and fourth order approximation in x, that in y is done in a similar manner.
\begin{enumerate}
\item First order: \textcolor{red}{Here, we must have $r=0$, $s=1$. We get the upwind scheme}
\begin{equation*}
  \delta_1 ^x f_{i,j}=f_{i,j}-f_{i-1,j}
\end{equation*}
Then, we have
\begin{equation*}
 \dpar{f}{x}(x_i,y_j)=\frac{1}{\Delta x}(f_{i,j}-f_{i-1,j}) +c_{0,1}\Delta x\frac{\partial^{2}f}{\partial x^{2}}(x_i,y_j)+\mathcal{O} ((\Delta x)^{2})   
\end{equation*}
and the flux is given by
\begin{equation*}
 \hat{f}_{i+\frac{1}{2},j} =f_{i+1,j} . 
\end{equation*}
\item \textcolor{red}{Second order case. Two cases are possible}{\color{red}
\begin{itemize}
\item $r=s=1$: centered case
$$\delta_2^xf_{i,j}=\frac{1}{2}\big ( f_{i+1,j}-f_{i-1,j}\big )$$
and the flux is
$$\hat{f}_{i+1/2,j}=\frac{f_{i+1,j}+f_{i,j}}{2}.$$
\item $r=0$, $s=2$. There
$$\delta_2^xf=\frac{3}{2}f_{i,j}-2f_{i-1,j}+\frac{1}{2}f_{i-2,j}$$
and the flux is 
$$\hat{f}_{i+1/2,j}=\frac{3}{2}f_{i,j}-\frac{1}{2}f_{i-1,j}.$$
\end{itemize}
The centered case will no longer be considered.}
\item Third order: \textcolor{red}{Only  choice is  possible $r=1$, $s=2$  and  we get}
\begin{equation*}
  \delta_3 ^x f_{i,j}=\frac{f_{i-2,j}}{6}-f_{i-1,j}+\frac{f_{i,j}}{2}+\frac{f_{i+1,j}}{3}
\end{equation*}
Therefore we have
\begin{equation*}
 \dpar{f}{x}(x_i,y_j)=\frac{1}{\Delta x}(\frac{f_{i-2,j}}{6}-f_{i-1,j}+\frac{f_{i,j}}{2}+\frac{f_{i+1,j}}{3}) +c_{2,1}(\Delta x)^3\frac{\partial^{4}f}{\partial x^{4}}(x_i,y_j)+\mathcal{O} ((\Delta x)^{4})   
\end{equation*}
As before, the flux is given by
\begin{equation*}
 \hat{f}_{i+\frac{1}{2},j} =-\frac{f_{i-1,j}}{6}+\frac{5}{6}f_{i,j} +\frac{f_{i+1,j}}{3} .
\end{equation*}
\item Fourth order: we consider two cases
\begin{itemize}
\item If $r=s=2$, we have
\begin{equation*}
  \delta_4 ^x f_{i,j}=\frac{f_{i-2,j}}{12}-\frac{2}{3}f_{i-1,j}+\frac{2}{3}f_{i+1,j}-\frac{f_{i+2,j}}{12}
\end{equation*}   
We can write
\begin{equation*}
 \dpar{f}{x}(x_i,y_j)=\frac{1}{\Delta x}(\frac{f_{i-2,j}}{12}-\frac{2}{3}f_{i-1,j}+\frac{2}{3}f_{i+1,j}-\frac{f_{i+2,j}}{12}) +c_{2,2}(\Delta x)^4\frac{\partial^{5}f}{\partial x^{5}}(x_i,y_j)+\mathcal{O} ((\Delta x)^{5})   
\end{equation*}
For the flux, we can write as follows
\begin{equation*}
 \hat{f}_{i+\frac{1}{2},j} =\frac{f_{i-1,j}}{12}+\frac{3}{4}f_{i,j} +\frac{3}{4}f_{i+1,j}+\frac{f_{i+2,j}}{12}.
\end{equation*}
\item If $r=1$ and $s=3$, we have
\begin{equation*}
  \delta_4 ^x f_{i,j}=-\frac{f_{i-3,j}}{12}+\frac{f_{i-2,j}}{2}-\frac{3}{2}f_{i-1,j}+\frac{5}{6}f_{i,j}+\frac{f_{i+1,j}}{4}
\end{equation*}   
Hence
\begin{equation*}
 \dpar{f}{x}(x_i,y_j)=\frac{1}{\Delta x}(-\frac{f_{i-3,j}}{12}+\frac{f_{i-2,j}}{2}-\frac{3}{2}f_{i-1,j}+\frac{5}{6}f_{i,j}+\frac{f_{i+1,j}}{4}) +c_{1,3}(\Delta x)^4\frac{\partial^{5}f}{\partial x^{5}}(x_i,y_j)+\mathcal{O} ((\Delta x)^{5})   
\end{equation*}
The flux becomes
\begin{equation*}
 \hat{f}_{i+\frac{1}{2},j} =\frac{f_{i-2,j}}{12}-\frac{5}{12}f_{i-1,j} +\frac{13}{12}f_{i,j}+\frac{f_{i+1,j}}{4}.
\end{equation*}
\textcolor{red}{We will only use the case $r=1$, $s=3$ because the case $r=2$, $s=2$ is centered.}
\end{itemize}
\end{enumerate}
\subsection{Limitation}\label{sec:limit}
We have explored two ways to introduce a nonlinear stabilisation mechanism. The first one is inspired from \cite{sweby,Leveque,Yee} and consists in "limiting" the flux, the second one is inspired by the MOOD paradigm \cite{clain,Vilar}.
\subsubsection{Limitation of the flux}
In this section, we only consider the space discretisation in x, that in y is done in a similar manner.
First, we calculate the difference between first and second order approximation in x. We have
\begin{equation*}
\delta_2 ^x f_{i,j}-\delta_1 ^x f_{i,j}=\frac{1}{6}(f_{i-2,j}-3f_{i,j}+2f_{i+1,j})
\end{equation*}
So for limitation
\begin{align*}
\widetilde{\delta_2^x} f_{i,j}&=\delta_1 ^x f_{i,j}+\frac{\theta}{6}(\delta_2 ^x f_{i,j}-\delta_1 ^x f_{i,j})\\
&=\delta_1 ^x f_{i,j}+\frac{\theta}{6}(f_{i-2,j}-3f_{i,j}+2f_{i+1,j})\\
&=\delta_1 ^x f_{i,j} (1+\frac{\theta r}{6})
\end{align*}
where $\theta \in [0,1]$ and
\begin{equation*}
r=\frac{(f_{i-2,j}-3f_{i,j}+2f_{i+1,j})}{f_{i,j}-f_{i-1,j}} =1-\frac{-2f_{i+1,j}+4f_{i,j}-f_{i-1,j}-f_{i-2,j}}{f_{i,j}-f_{i-1,j}}  
\end{equation*}
We can see that if $f$ is smooth, $r \approx 0$. To have a monotonicity condition, we also need
\begin{equation*}
    6+\theta r \geq 0
\end{equation*}
We want to find $\theta(r)$ such that $\theta(0)=1$ and 
$$0\leq 1+\frac{\theta(r)r}{6} \leq M$$
where $M$ is a constant that will dictate the CFL constraint. Since $\theta(0)=1$, we observe that $M\geq 1$  is a necessary condition. One solution is to choose a $\alpha \in ]0,\min\{6,6(M-1) \}[$. We take
\begin{equation*}
    \theta(r) =  \begin{cases}
    1 & \text{if }r \in[-6+\alpha,6(M-1)-\alpha] , \\
    \frac{-6+\alpha}{r} & \text{if }r <-6+\theta, \\
    \frac{6(M-1)-\alpha}{r} &\text{if } r> 6(M-1)+\alpha.
  \end{cases}
\end{equation*}
After calculations, we get
\begin{equation*}
   \widetilde{\delta_2 ^x}=\delta_1 ^x  (1+\theta r)=\delta_1 ^x+\psi(\delta_1 ^x,\delta_2 ^x)(\delta_2 ^x-\delta_1 ^x)
\end{equation*}
where
\begin{equation*}
    \psi(r,s)(s-r) =  \begin{cases}
    s-r & \text{if }\frac{s}{r} \in[-6+\alpha,6(M-1)-\alpha] , \\
    (-6+\alpha) r & \text{if }\frac{s}{r} <-6+\theta, \\
    (6(M-1)-\alpha) r & \text{if } \frac{s}{r}> 6(M-1)+\alpha.
  \end{cases}
\end{equation*}
For the first and fourth order approximations, it is possible to follow the same technique.

The main drawback of this approach is that the projection on the conservative variable plays no role, while the variable we are really interested in are these variables \ldots so that we can expect some weaknesses.

\subsubsection{Stabilisation by the MOOD technique}
This is inspired from \cite{clain,Vilar}, the variables on which we test are the physical variables $(\rho, \bv, p)$. This section also justifies the way that have written the update of the spatial term as in \eqref{residual}, i.e. a sum of residual that are evaluated on the elements $K_{i+1/2.j+1/2}=[x_i,x_{i+1}]\times [y_j, y_{j+1}]$.

Starting from $\bF^n$, we first compute $\widetilde{\bF^{n+1}}$ with the full order method (i.e. full order in time and space). In the numerical examples, we will take the fourth order accurate method, but other choices can be made.
The algorithm is as follow: We define a vector of logical $\texttt{Flag}_p$ which is false initially for all the grid points, and a vector of logical $\texttt{Flag}_e$ which is also set to false.
\begin{enumerate}
\item For each mesh point $(x_i,y_j)$, we compute  $\tilde{V}_{ij}=\widetilde{(\rho, \bv, p)}^{n+1}$ variables defined by $\P\widetilde{\bF^{n+1}}$. If  $\Tilde{\rho}^{n+1}_{ij}\leq 0$  or $\Tilde{p}^{n+1}_{ij}\leq 0$ or one of the components of $\tilde{V}_{ij}$ are NaN \footnote{$x$ is not a number if $x\neq x$.}, we set $\texttt{Flag}_p(i,j)=.true.$
\item Then we loop over the quads $[x_i,x_{i+1}]\times [y_j, y_{j+1}]$. If for one of the 4 corners $(x_l,y_q)$, $\texttt{Flag}_p(l,q)=.true.$, we set  $\texttt{Flag}_e(i+1/2,j+1/2)=.true.$
\item  For each element such that $\texttt{Flag}_e(i+1/2,j+1/2)=.true.$, we recompute, for each sub-time step the four residuals defined by \eqref{residual:2}
\end{enumerate}

In \cite{clain,Vilar} is also a way to detect local extremas, and in \cite{Vilar} to differenciate the local smooth extremas from the discontinuities. We have not use this here, and there is way of improvement.
The first order version of our scheme amounts to global Lax-Friedrichs. To make sure that the first order scheme is domain invariant preserving, we can apply this procedure to each of the cycle of the DeC procedure, we have not done that in the numerical examples.

\section{Stability analysis}
\label{kinetic:sec_stability}
We will study the stability of the discretisation of the homogeneous problem.
As discussed in the previous section, since the matrices $\Lambda_x$  and $\Lambda_y$ are diagonal, it is enough to consider again the following transport equation
\begin{equation}
\dpar{f}{t} +a\dpar{f}{x} +b\dpar{f}{y}=0 ,
\end{equation}
Since $ab=0$ in the case of four waves model, we have the same results as one dimensional case \cite{Abgrall}. In other cases, 
we perform the Fourier analysis to evaluate the amplification factors of the method, first without
defect correction iteration, then with defect correction iteration. We
assume, without loss of generality, that $a,b>0$. we denote Fourier symbol of $\delta^x$ and $\delta^y$ as $g_1$ and $g_2$, respectively. For $\delta^x$ and $\delta^y$  operators, we considered four cases in the previous section, we have
\begin{itemize}
 \item First order in both $x$ and $y$: we have $g^{(1)}(\theta)=1-e^{i\theta}$ and $g_1=g^{(1)}(\theta_1)$, $g_2=g^{(1)}(\theta_2)$
  We see that $\Re(g^{(1)})\geq 0$ and $\max_\theta\vert g^{(1)}\vert = 4$.
 \item Second order in $x$ and $y$: $g_1=g^{(2)}(\theta_1)$ and $g_2=g^{(2)}(\theta_2)$ where 
 $$ g^{(2)}(\theta)=\frac{3}{2}-2e^{-i\theta}+\frac{e^{-2i\theta}}{2}.$$
 We notice that $\Re(g^{(2)})=\big ( \cos\theta-1\big )^2\geq 0$ and $\max_\theta\vert g^{(2)}\vert = 2$.
 \item Third order in both $x$ and $y$: $g_1=g^{(3)}(\theta_1)$ and $g_2=g^{(3)}(\theta_2)$ where 
$$ g^{(3)}=\frac{e^{-2i\theta}}{6}-e^{-i\theta}+\frac{1}{2}+\frac{e^{i\theta}}{3}.$$
Again, $\Re(g^{(3)})=\frac{1}{3}\big (\cos\theta-1\big )^2\geq 0$ and $\max_\theta\vert g^{(3)}\vert = \tfrac{3}{2}$
 \item Fourth order in both $x$ and $y$:  we only  consider the  case $r=1$, $s=2$. We have $g_1=g^{(4)}(\theta)$ and $g_2=g^{(4)}(\theta_2)$ where
 $$
 g^{(4)}(\theta-1)=-\frac{e^{-3i\theta}}{12}+\frac{e^{-2i\theta}}{2}-\frac{3}{2}e^{-i\theta}+\frac{5}{6}+\frac{e^{i\theta}}{4}.$$
 and we have $\Re(g^{(4)})=\frac{1}{3}\big (1-\cos\theta\big )^3\geq 0$ and $\max_\theta\vert g^{(4)}\vert = \frac{8}{3}$
\end{itemize}
Now, we consider first, second and fourth order approximations in time in the $L_2$ operator. In the sequel, we 
set $g=\mu g_1+\nu g_2$ with $\mu=a\frac{\Delta t}{\Delta x}$ and $\nu=b\frac{\Delta t}{\Delta y}$.
\begin{enumerate}
\item First order in time: Using Fourier transform, the $L_2$ operator can be
written as follows
\begin{equation*}
 \hat{f}^{n+1} -\hat{f}^{n}+\mu g_1  \hat{f}^{n+1}+\nu g_2  \hat{f}^{n+1}=0.
\end{equation*}
 The  amplification factor is 
$$G_1=\frac{1}{1+g}$$
In order to have a stable scheme for the first order scheme, we should have $\lvert G\rvert \leq 1$,  and a necessary and sufficient condition is  $2\Re(g)+|g|^2\geq 0$.

The defect correction iteration is written as 
\begin{equation*}
 \hat{f}^{(r+1)} =\hat{f}^{n}-\mu g_1  \hat{f}^{(r)}-\nu g_2  \hat{f}^{(r)}
\end{equation*}
The resulting formula for the amplification factor $G_{r+1}$ is given by
\begin{align*}
 G_{1,0}&=1 \\
 G_{,1r+1}(g)&=1-g G_{1,r}(g)
\end{align*}
We can observe that
\begin{equation}\label{G:1}
 G_{1,r+1}(g)-G_1(g)= (-1)^{r+1}g^{r+1}  \big (1-G_1(g)\big ).
\end{equation}
We note that $g^{r+1}\rightarrow 0$ if $\vert g\vert<1$.
\item Second order in time: We again use Fourier transform, and write $L_2$ as
\begin{equation*}
 \hat{f}^{n+1} -\hat{f}^{n}+\frac{\mu}{2}( g_1  \hat{f}^{n}+g_1  \hat{f}^{n+1})+\frac{\nu}{2} (g_2  \hat{f}^{n}+g_2  \hat{f}^{n+1})=0
\end{equation*}
In this case the amplification factor is
$$G_2=\frac{1-\frac{g }{2}}{1+\frac{g }{2} }$$
We conclude that under the following condition stability holds
$$\text{Re}(g)\geq 0$$
The defect correction iteration reads
\begin{equation*}
 \hat{f}^{(r+1)}= \hat{f}^{n}-\frac{\mu}{2}( g_1  \hat{f}^{n}+g_1  \hat{f}^{(r)})-\frac{\nu}{2} (g_2  \hat{f}^{n}+g_2  \hat{f}^{(r)})
\end{equation*}
we have
\begin{align*}
 G_{2,0}&=1 \\
 G_{2,r+1}(g)&=1-\frac{g}{2}-\frac{g}{2}G_{2,r}(g)
\end{align*}
It is easy to check that
\begin{equation}\label{G:2}
 G_{2,r+1}(g)-G_3(g)= (-1)^{r+1}\big (\frac{g}{2}  \big )^{r+1}  \big (1-G_2(g)\big ).
\end{equation}
We note that $\big (\frac{g}{2}  \big )^{r+1}\rightarrow 0$ if $\vert g\vert g\leq 2$.
\item Fourth order in time: Similarly, we have the following formula for $L_2$ operator
\begin{equation}
\label{eq:23}
\begin{split}
 \hat{f}^{n+\frac{1}{2}}&-\hat{f}^{n}+\mu (\frac{5}{24}g_1 \hat{f}^{n}+\frac{1}{3}g_1 \hat{f}^{n+\frac{1}{2}}-\frac{1}{24}g_1 \hat{f}^{n+1})+\nu(\frac{5}{24}g_2 \hat{f}^{n}+\frac{1}{3}g_2 \hat{f}^{n+\frac{1}{2}}-\frac{1}{24}g_2 \hat{f}^{n+1})=0\\
  \hat{f}^{n+1}&-\hat{f}^{n}+\mu (\frac{1}{6}g_1 \hat{f}^{n}+\frac{2}{3}g_1 \hat{f}^{n+\frac{1}{2}}+\frac{1}{6}g_1 \hat{f}^{n+1})+\nu(\frac{1}{6}g_2 \hat{f}^{n}+\frac{2}{3}g_2 \hat{f}^{n+\frac{1}{2}}+\frac{1}{6}g_2 \hat{f}^{n+1})=0
  \end{split}
\end{equation}
We can rewrite \eqref{eq:23} in matrix form
\begin{equation*}
 \begin{pmatrix}
	\hat{f}^{n+\frac{1}{2}}  \\
	\hat{f}^{n+1}	
\end{pmatrix}=G_4\begin{pmatrix} 
\hat{f}^{n}\\ \hat{f}^{n}
\end{pmatrix}   
\end{equation*}
where, setting $\theta=24+12g+2g^2$
\begin{equation*}
\begin{split}
 G_4(g)&=
 \begin{pmatrix}
	1 +\frac{g }{3} &-\frac{g }{24}\\
	\frac{2g }{3}	& 1 +\frac{g }{6}
\end{pmatrix}^{-1}\begin{pmatrix} 
1-\frac{5g }{24}\\ 1-\frac{g}{6}
\end{pmatrix}\\ 
&\quad =\frac{1}{2\theta}
\begin{pmatrix} 
{24-g^2}\\
{24-12g+2g^2}\\
\end{pmatrix}
\end{split}
\end{equation*}
In order to have a stable scheme, one should have $\max\{\lvert G_1\rvert,\lvert G_2\rvert\}\leq 1$.
The defect correction iteration reads
\begin{equation}
\label{eq:24}
\begin{split}
 \hat{h}_1^{(r+1)}&=\hat{f}^{n}-\mu (\frac{5}{24}g_1 \hat{f}^{n}+\frac{1}{3}g_1 \hat{h}_1^{(r)}-\frac{1}{24}g_1 \hat{h}_2^{(r)})-\nu(\frac{5}{24}g_2 \hat{f}^{n}+\frac{1}{3}g_2 \hat{h}_1^{(r)}-\frac{1}{24}g_2 \hat{h}_2^{(r)})\\
  \hat{h}_2^{(r+1)}&=\hat{f}^{n}-\mu (\frac{1}{6}g_1 \hat{f}^{n}+\frac{2}{3}g_1 \hat{h}_1^{(r)}+\frac{1}{6}g_1 \hat{h}_2^{(r)})-\nu(\frac{1}{6}g_2 \hat{f}^{n}+\frac{2}{3}g_2 \hat{h}_1^{(r)}+\frac{1}{6}g_2 \hat{h}_2^{(r)})
  \end{split}
\end{equation}
Rewriting \eqref{eq:24} in matrix form, we obtain
\begin{equation*}
\hat{h}^{(r+1)}=\begin{pmatrix}
   1- \frac{5g}{24}\\1- \frac{g }{6}
\end{pmatrix} -gM
\hat{h}^{(r)}
\end{equation*}
where $M=\begin{pmatrix}
  \frac{1}{3}  & \frac{-1}{24}\\ \frac{2}{3} & \frac{1}{6}
\end{pmatrix}$. We get
\begin{align*}
 G_{4,0}&=\begin{pmatrix}
   1 \\1 
 \end{pmatrix} \\
 G_{4,r+1}(g)&=\begin{pmatrix}
   1- \frac{5g }{24}\\1- \frac{g }{6}
   \end{pmatrix} -gM G_{4,r}(g)
\end{align*}
Hence
\begin{equation}\label{G:4}
 G_{4,r+1}(g)-G_4(g)= (-1)^{r+1}(g)^{r+1} M^{r+1} \Bigg(\begin{pmatrix}
   1 \\1 
 \end{pmatrix}-G_4(g) \Bigg).
\end{equation}
We note that $(-1)^{r+1}(g)^{r+1} M^{r+1}\rightarrow 0$ if $\vert g\vert \Vert M\Vert_2<1$, i.e. if $\vert g\vert \leq \big (\tfrac{337+\sqrt{104353}}{1152}\big )^{-1}\approx 1.745356304$.
\end{enumerate}
{\color{red}Using this relations, we have plotted the zone $\mathcal{A}_G$ where $\vert G\vert \leq 1$ for the second order scheme (Crank Nicholson with DeC iteration) and $\max(\vert G(1)\vert , \vert G(2)\vert) <1$. To get a stable scheme, we need that part of $\mathcal{A}_G$ has a non empty intersection with $\{(x,y), x\geq 0\}$. We have plotted on Figure \ref{A} $\mathcal{A}_G$ for the second order Dec with 3 and 4 iterations, and for the 4th order DeC with 4 and 5 iterations.
\begin{figure}[h]
\begin{center}
\subfigure[]{\includegraphics[width=0.45\textwidth]{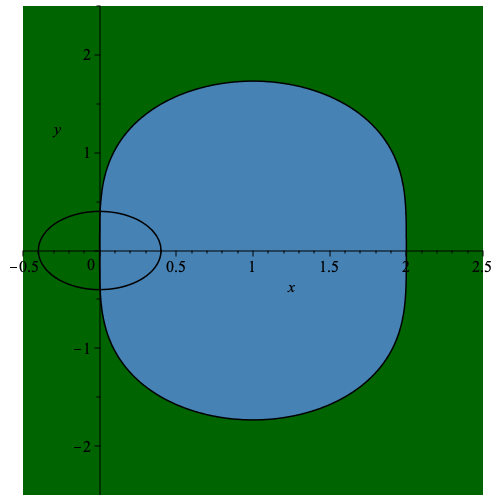}}
\subfigure[]{\includegraphics[width=0.45\textwidth]{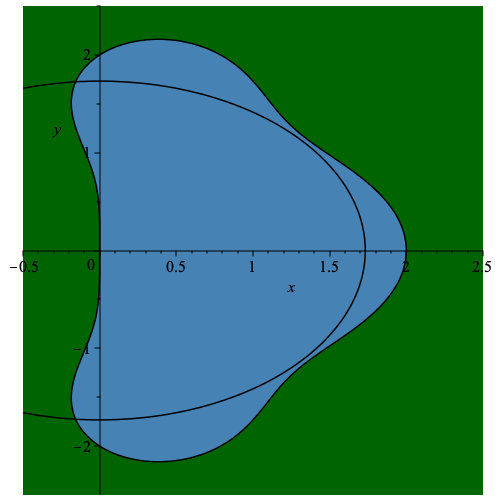}}
\subfigure[]{\includegraphics[width=0.45\textwidth]{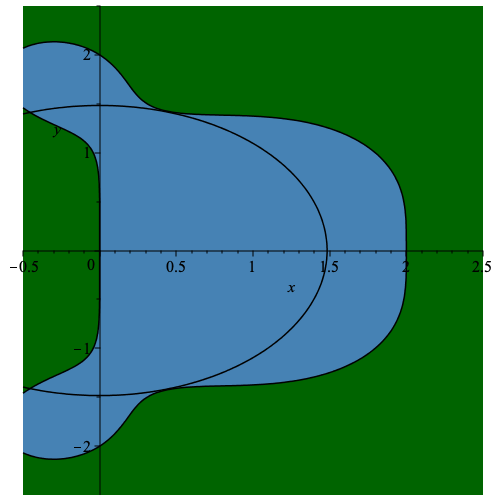}}
\subfigure[]{\includegraphics[width=0.45\textwidth]{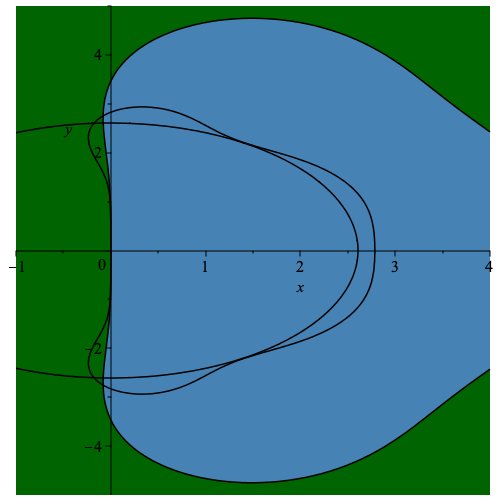}}
\subfigure[]{\includegraphics[width=0.45\textwidth]{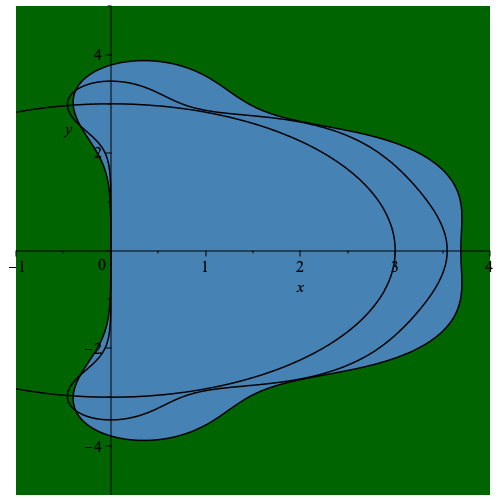}}
\subfigure[]{\includegraphics[width=0.45\textwidth]{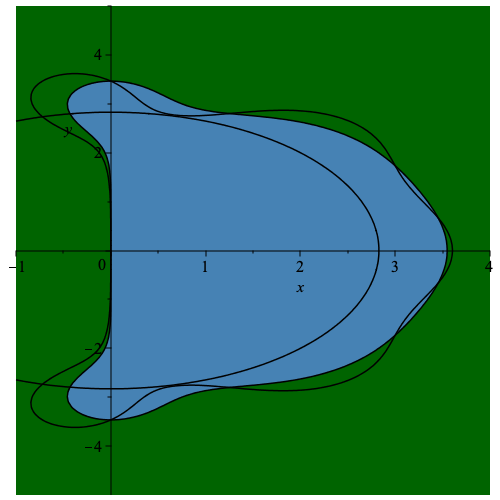}}
\end{center}
\caption{\label{A} Plots of the stability region   with the circle of radius $\rho$ for DeC 2nd order (a): $\rho_2^2=0.4$, 2 iterations, (b): $\rho_3^2=\sqrt{3}$, 3 iterations, (c): $\rho_4^2=\sqrt{2.2}$, 4 iterations and DeC 4th order (d): $\rho_4^4=\sqrt{6.8}$, 4 iterations, (e): $\rho_5^3=3$, 5 iterations, (f): $\rho_6^3=2\sqrt{2}$, 6 iterations.  }
\end{figure}
We see that the larger $\rho$ is achieved for $r+1$ iterations for scheme of order $r$.

From \eqref{G:1}-\eqref{G:4}, the amplification factor of the two dimensional scheme is $G_k^r\big (\nu g(\theta_1)+\mu g(\theta_2)\big )$, so we need
to check when $\vert \nu g(\theta_1)+\mu g(\theta_2)\vert\leq \rho_k^r$. Since $\vert \nu g(\theta_1)+\mu g(\theta_2)\vert\leq 2 \max(|\mu|, |\nu|)\max_\theta\vert g(\theta)\vert |$, we display $\max(|\mu|, |\nu|)$ in Table \ref{TableA}. 
\begin{table}[h]
\begin{center}
\begin{tabular}{|c||c|c|c||}
\hline
& DeC(2,2)& DeC(2,3)& DeC(2,4)\\
\hline
$g^{(2)}$&$0.4$&$2/\sqrt{3}\approx 1.15$&$\sqrt{2.2}\approx 1.58$ \\\hline
& DeC(4,4)& DeC(4,5)& DeC(4,6)\\ \hline
$g^{(4)}$ &$0.75\sqrt{6.8}\approx 1.95$ & $2.25$ & $1.5\sqrt{2}\approx 2.12$
\\
\hline
\end{tabular}
\end{center}
\caption{\label{TableA} CFL condition for the scheme DeC(r,p): r is the order, p the number of iteraion}
\end{table}
 Using this amplification factors, we obtain the maximum values of the CFL number for which the various options leads to stability. Even with the BGK source term, the schemes are always stable for $CFL=1+\delta$, $\delta>0$. For example, the combination 4th order in time/4th ordre in space is stable for $CFL\approx 1.3$, and this can be checked experimentally on the vortex case bellow, with periodic boundary conditions. For that case, we have not ben able to run higher that CFL 1.3, it may be an effect of the BGK source term. This results are also consistent with  \cite{Abgrall}.
 }
\section{Test cases}
\label{kinetic:sec_testcases}
In this section we will present the numerical results that illustrate the behavior of our scheme on scalar problems and Euler equations. To evaluate the accuracy and robustness of the proposed high order asymptotic preserving kinetic scheme, in the following 
we perform the convergence analysis for the scalar problems and study several benchmark problems for the Euler equations of gas dynamics.
\subsection{Scalar problems}
Consider the two-dimensional advection equation:
\begin{equation*}
\frac{\partial u}{\partial t}+ \dpar{u}{x}+ \dpar{u}{y}  =0,~ (x,y,t)\in [-2,2]\times[-2,2]\times \mathbb{R}^{+},
\end{equation*}
and periodic boundary conditions. We consider the following initial condition:
\begin{equation*}
u_0(x,y)=\sin(\pi x+\pi y),~ (x,y) \in (-2,2) \times (-2,2).
\end{equation*}
The CFL number is set to 1. The convergence for the density is shown in Tables \ref{tab:adv_ord2} and \ref{tab:adv_ord4} for  final time $T = 10$ for orders 2 and 4, which result in the predicted convergence rates of second and fourth order, respectively.





\begin{table}[H]
  \centering
  \caption{Convergence study for the advection equation for order 2 at $T=10$.} 
\begin{tabular}{|c|c c c c  c c |} 
\hline
h & $L^1$-error &slope &  $L^2$-error& slope& $L^{\infty}$-error & slope\\
\hline
0.05& $1.0698.10^{+1}$  & - &  $2.8479.10^{0}$ &- &$9.3878.10^{-1}$ &  -   \\

0.025& $3.5595.10^{0}$  & 1.59 & $9.6212.10^{-1}$  &1.57 & $3.3039.10^{-1}$&  1.51   \\

0.0125& $6.8578.10^{-1}$  & 2.38 & $1.8812.10^{-1}$  & 2.35&$6.5662.10^{-2}$ & 2.33    \\

0.00625& $1.4701.10^{-1}$  & 2.22 &  $4.0558.10^{-2}$ &2.21 &$1.4243.10^{-2}$ &  2.20   \\

0.003125& $3.4890.10^{-2}$
  & 2.08 & $9.6578.10^{-3}$  &2.07 & $3.4037.10^{-3}$&  2.07   \\
\hline
\end{tabular}
	\label{tab:adv_ord2}
\end{table}

\begin{table}[H]
  \centering
  \caption{Convergence study for the advection equation for order 4 at $T=10$.} 
\begin{tabular}{|c|c c c c  c c |} 
\hline
h & $L^1$-error &slope &  $L^2$-error& slope& $L^{\infty}$-error & slope\\
\hline
0.05&$4.7601.10^{0}$ &   -&$1.2919.10^{0}$  & -  &$4.1702.10^{-1}$ &     - \\

0.025&$3.1678.10^{-1}$ & 3.91  & $8.5482.10^{-2}$ & 3.92  &$2.9212.10^{-2}$ &   3.84     \\

0.0125&  $1.8698.10^{-2}$ & 4.08 &  $5.1232.10^{-3}$ &4.06 &$1.7850.10^{-3}$ & 4.03    \\

0.00625& $1.1427.10^{-3}$  & 4.03 & $3.1527.10^{-4}$  & 4.02&$1.1072.10^{-4}$ & 4.01    \\

0.003125& $7.0804.10^{-5}$  & 4.01 &   $1.9599.10^{-5}$& 4.01&$6.9070.10^{-6}$ &4.00     \\
\hline
\end{tabular}
	\label{tab:adv_ord4}
\end{table}
\subsection{Euler equations}
In this section we first test our scheme on the following Euler equations in 2D
\begin{equation}
\label{results_euler}
\frac{\partial \bu}{\partial t}+ \dpar{\bA_1(\bu)}{x}+ \dpar{\bA_2(\bu)}{y}  =0.
\end{equation}
where
$$\bu=(\rho, \rho \bv, E) \text{ and }\bA=(\rho \bv, \rho \bv\otimes\bv+p\text{Id }, (E+p)\bv)=(\bA_1,\bA_2),$$
We run some standard cases: the isentropic case,  the Sod problem and a strong shock tube.
\subsubsection{Isentropic vortex}
The first considered test case is the isentropic case. The boundary conditions are periodic.
The initial conditions are given by
\begin{align*}
\rho &= \left[ 1 - \frac{(\gamma -1)\beta^2}{32\gamma \pi^2}\exp \big(1 - r^2\big) \right]^{\frac{1}{\gamma-1}}, \\
v_x &= 1 - \frac{\beta}{4\pi}\exp\left(\frac{1-r^2}{2}\right)(y-y_c), \\
v_y &= \frac{\sqrt{2}}{2}+ \frac{\beta}{4\pi}\exp\left(\frac{1-r^2}{2}\right)(x-x_c), \\
p &= \rho^\gamma,
\end{align*}
where $\gamma = 1.4$, $\beta=5$ and $r=\sqrt{(x-x_c)^2+(y-y_c)^2}$. The computational domain is a square $[-10, 10]\times[-10,10]$. Also, the free stream conditions are given by:
$$\rho_{\infty}=1,\quad v_{x,\infty}=1,\quad v_{y,\infty}=\frac{\sqrt{3}}{2},\quad p_{\infty}=1.
$$
The $y$-velocity is chosen such that the particle trajectories never coincide to a mesh point, if one start from a grid point initially.
The final time is first set to $T=5$ for a convergence study (because of the cost mostly). The center of the vortex is set in $(x_c,y_c)=(Tv_{x,\infty},Tv_{y,\infty})$, modulo $20$.

The reference
solution is obtained on a regular Cartesian mesh consisting of $100\times 100$ elements and 4th order scheme in space and time. The CFL number is set to 1, and we consider the four waves model. In Figs. \ref{fig:vortex_p}, \ref{fig:vortex_density} and \ref{fig:vortex_velocity}, we have displayed the pressure, density and norm of the velocity at $T=5$, respectively. We can observe that the plots show a very good behavior of the numerical
scheme. The obtained convergence curves for the three
 considered error norms are displayed in Fig. \ref{fig:vortex_error} and show an excellent correspondence to the theoretically
predicted order.
\begin{figure}[H]
\begin{center}
\subfigure[4-th order scheme in space and time]
{\includegraphics[width=0.45\textwidth]{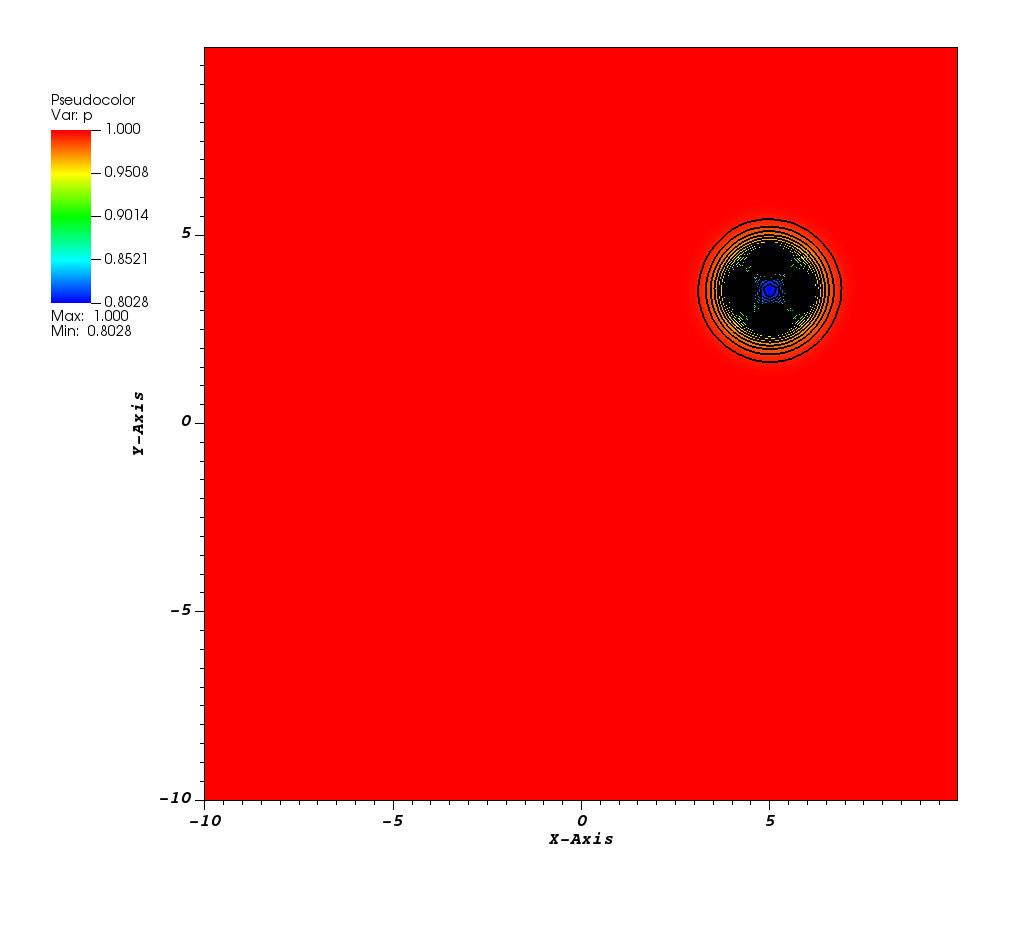}}
\subfigure[Exact]
{\includegraphics[width=0.45\textwidth]{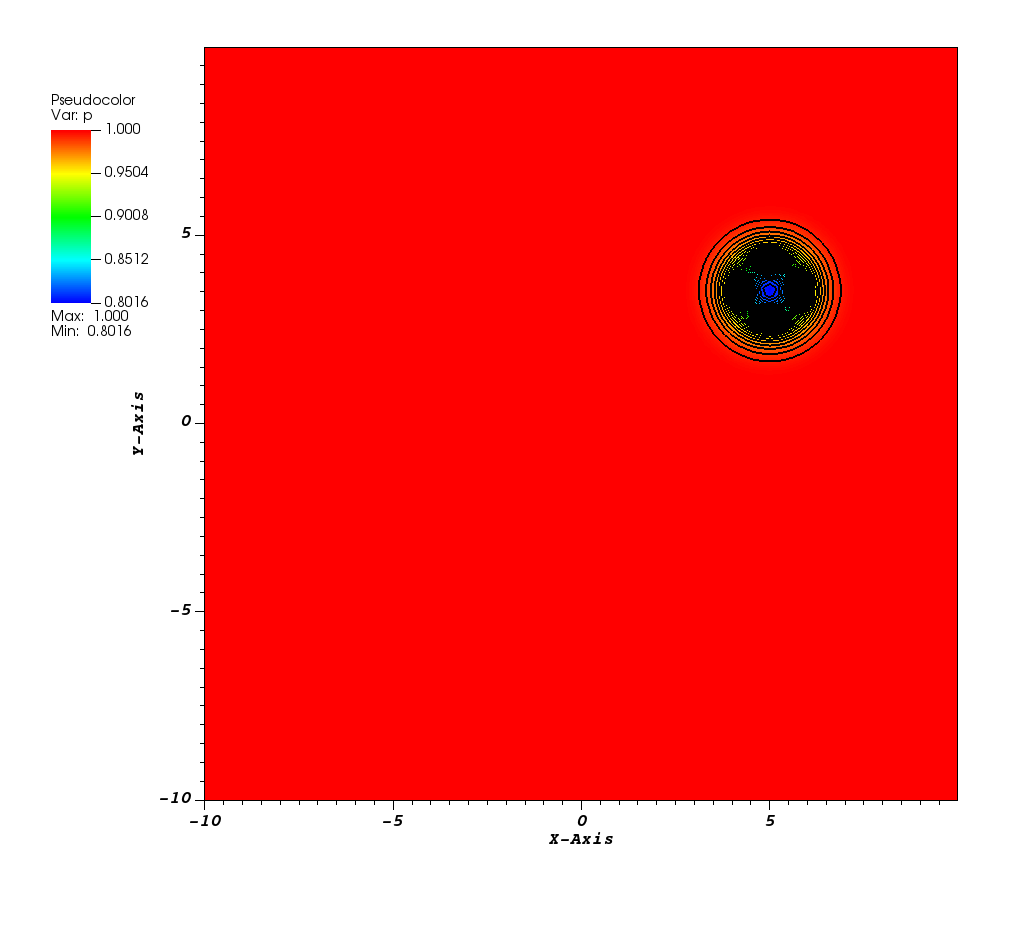}}
\end{center}
\caption{Plot of the pressure for the vortex problem at $T=5$.}
\label{fig:vortex_p}
\end{figure}

\begin{figure}[H]
\begin{center}
\subfigure[4-th order scheme in space and time]
{\includegraphics[width=0.45\textwidth]{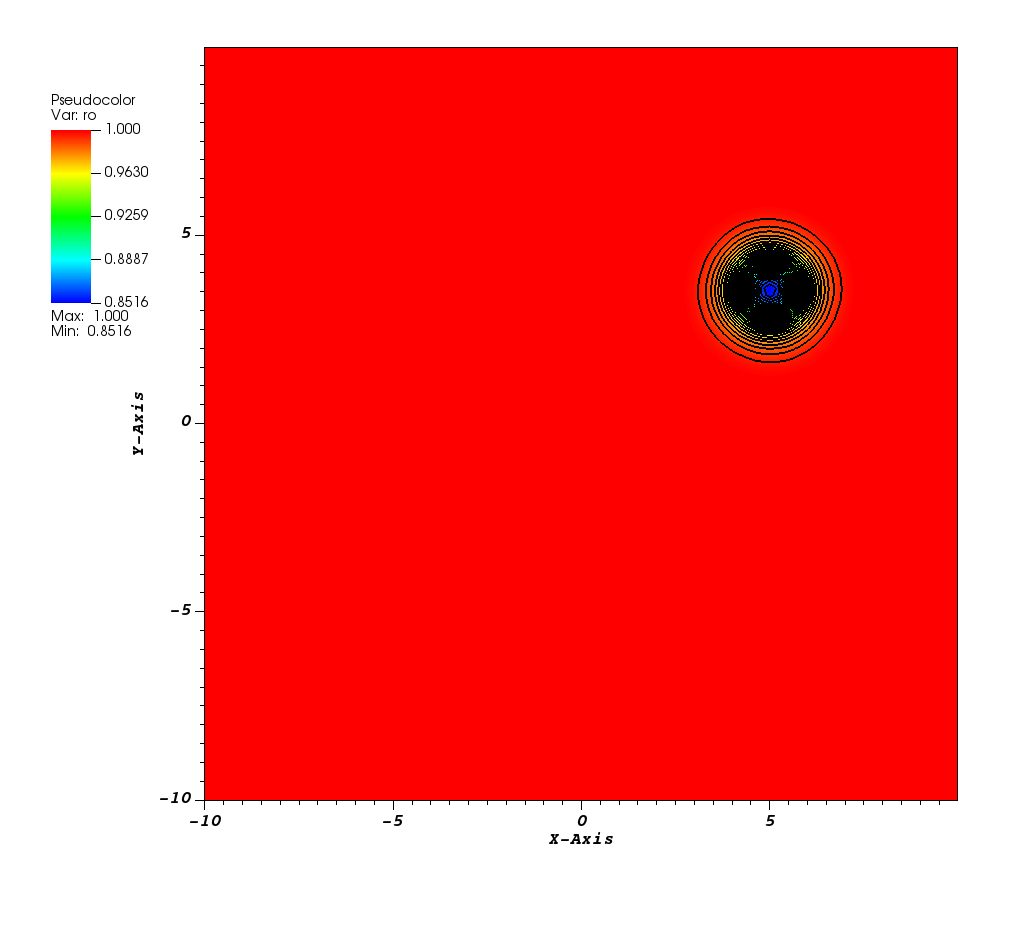}}
\subfigure[Exact]
{\includegraphics[width=0.45\textwidth]{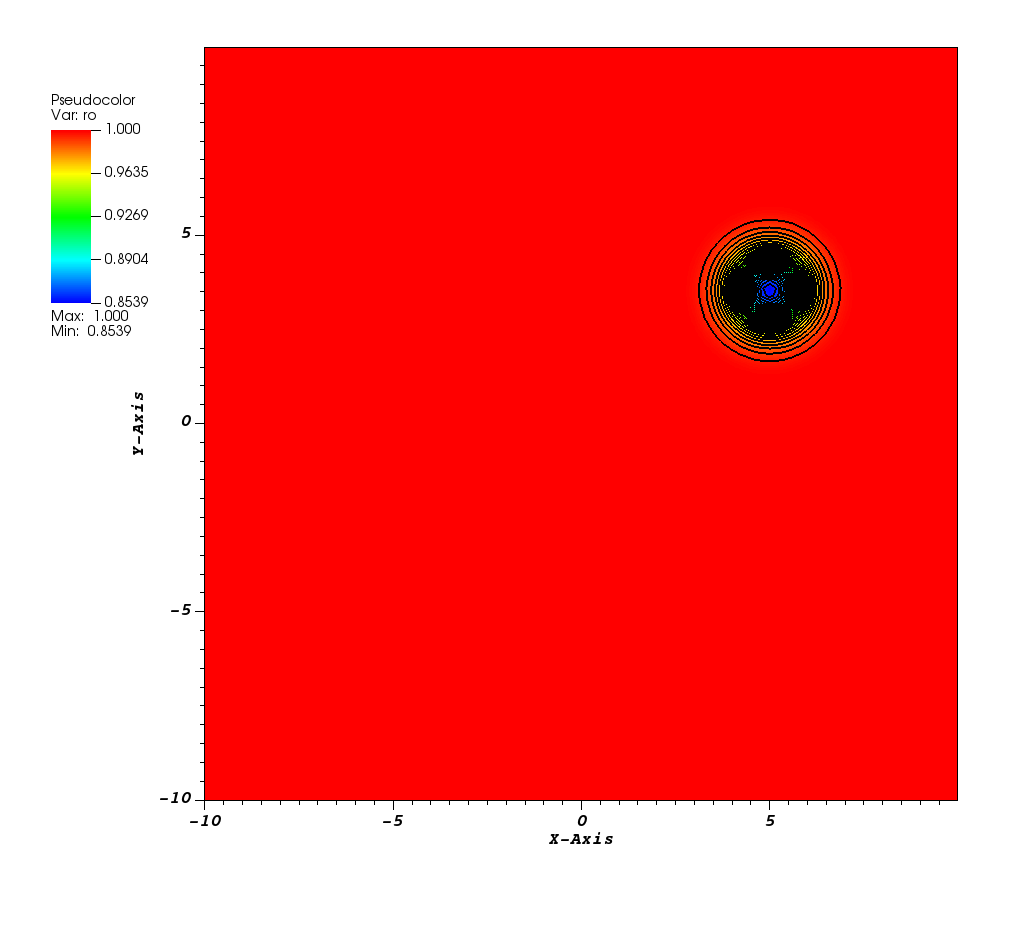}}
\end{center}
\caption{Plot of  the density for the vortex problem at $T=5$.}
\label{fig:vortex_density}
\end{figure}

\begin{figure}[H]
\begin{center}
\subfigure[4-th order scheme in space and time]
{\includegraphics[width=0.45\textwidth]{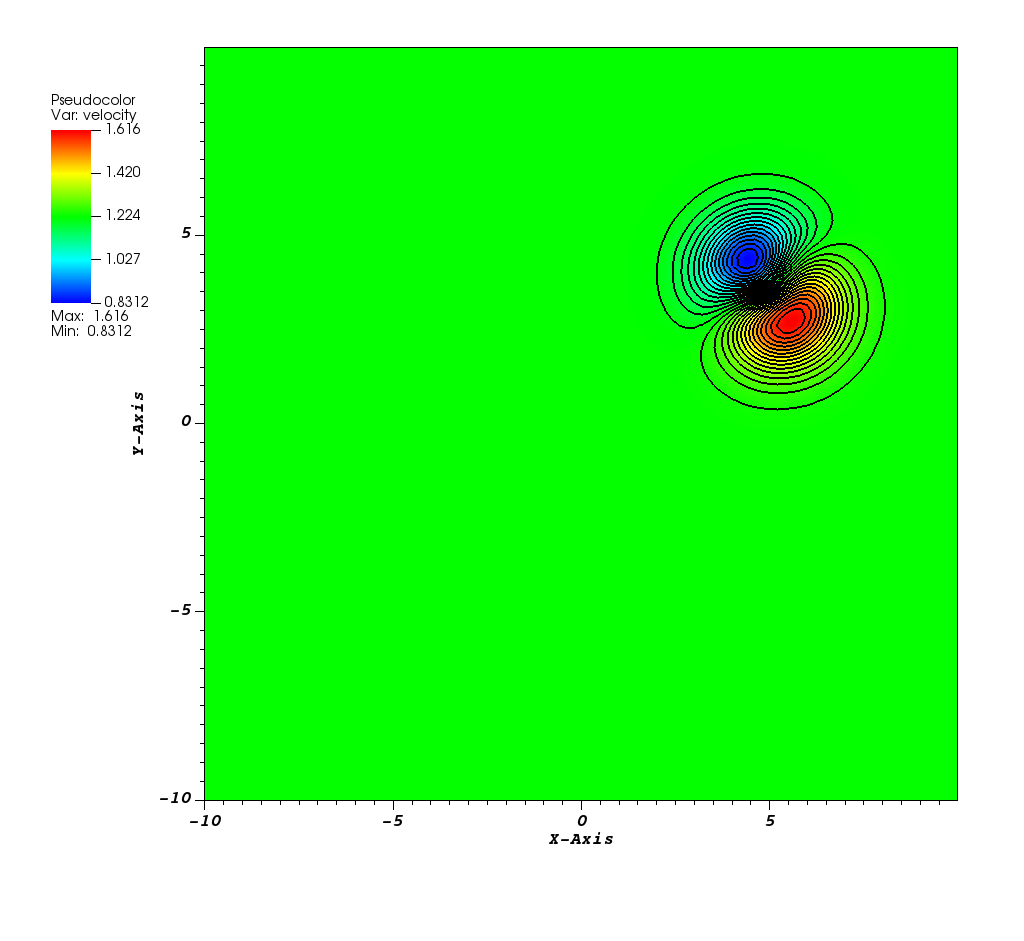}}
\subfigure[Exact]
{\includegraphics[width=0.45\textwidth]{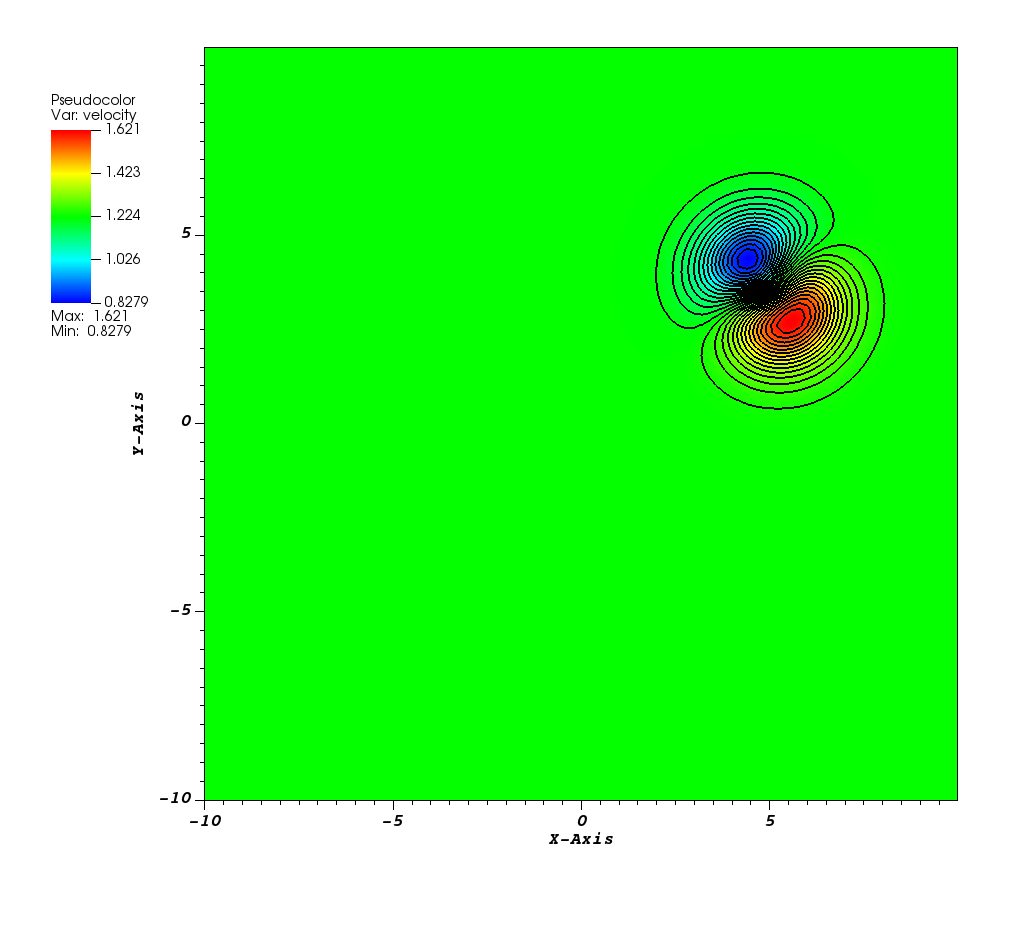}}
\end{center}
\caption{Plot of the  norm of the velocity for the vortex problem at $T=5$.}
\label{fig:vortex_velocity}
\end{figure}

\begin{figure}[H]
\begin{center}
{\includegraphics[width=0.45\textwidth]{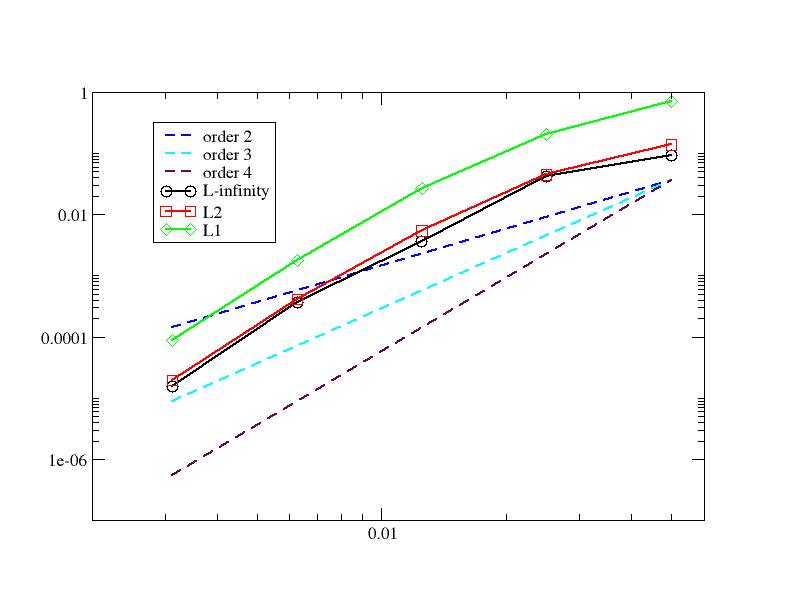}}
\end{center}
\caption{Convergence plot of density for the fourth order scheme in space and time  at $T=5$.}
\label{fig:vortex_error}
\end{figure}

In order to illustrate the long time behavior of the scheme, we show the pressure for $T=200$  and the error between the computed pressure and the exact one on Fig. \ref{vortex:200} and a $200\times 200$ grid. Note that the typical time for a vortex to travel across the domain is about $10$.
\begin{figure}[H]
\begin{center}
    \subfigure[$p$]{\includegraphics[width=0.45\textwidth]{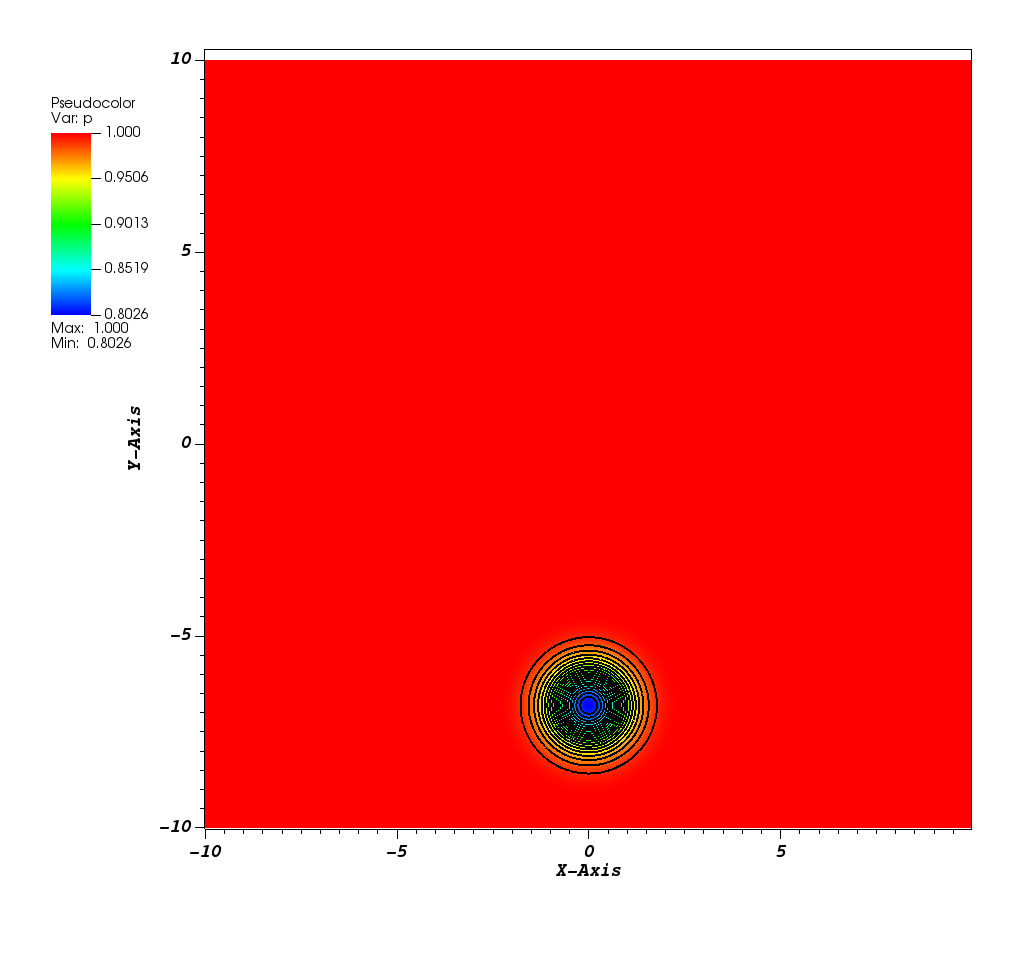}}
    \subfigure[error]{\includegraphics[width=0.45\textwidth]{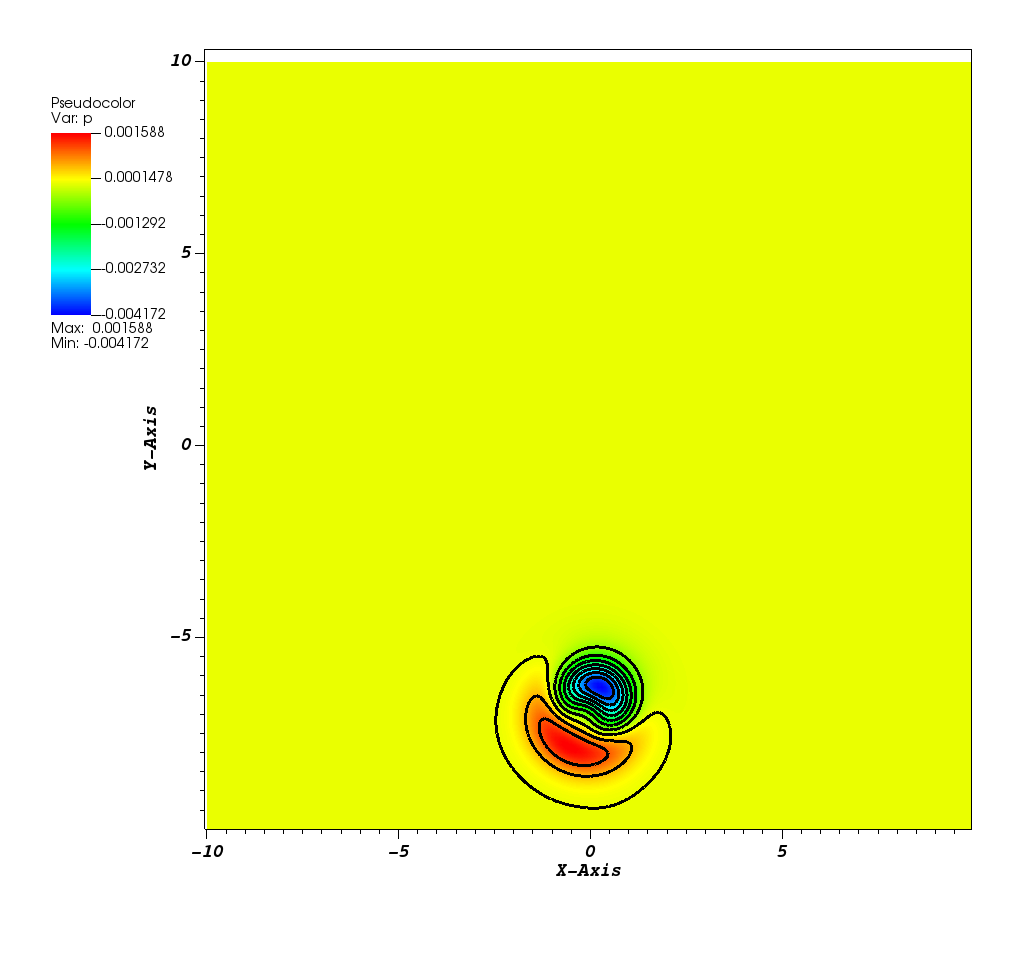}}
\end{center}
\caption{\label{vortex:200} Pressure and error between the computed solution and the exact one at $T=200$ on a $200\times 200$ grid. We have $p_{i,j}-p^{ex}_{i,j}\in [-4.2\; 10^{-3}, 1.6\; 10^{-3}]$.}
\end{figure}
{\color{red}
\begin{remark}[About the stability condition]
In this paper, we have focussed our attention to simulations with CFL=1. However, the stability analysis suggests that higher CFL can be used. In the case of the vortex case, using 5 iteration, we have been able to run this case, up to $T=200$ with CFL=1.2. This is smaller than what is suggested by table \ref{stabilityDeC}. In this table, only the convection operator is considered, and we are not able to make an analysis where the source term is also included. It seems that the constraints are more severe than those suggested by the linear stability analysis.
\end{remark}
}
\subsubsection{Sod test case}
Further, we have tested our high order kinetic scheme on a well-known 2D Sod benchmark problem.  This test is again solving Euler
equation \eqref{results_euler}. The domain is a square $[-1,1]\times [-1,1]$. The initial conditions
are given by
\begin{equation*}
    (\rho _0,v_{x,0},v_{y,0},p_0) =  \begin{cases}
    ( 1 , 0 , 0 , 1 ),& \text{if } r \leq 0 . 5 , \\
    ( 0 . 125 , 0 , 0 , 0 . 1 ),& \text{otherwise.}
  \end{cases}
\end{equation*}
 and the boundary conditions are periodic. The final time is $T=0.16$  and the CFL number is set to 1. The two stabilisation methods have been tested and compared. The results for the limitation method are in  Fig. \ref{fig:limit}, while the ones obtained with the MOOD method are displayed on Fig. \ref{fig:mood}. The two methods provide almost identical results. However, the MOOD method, for this case, never activates the first order scheme, hence the results are obtained with the 4th order scheme. One can observe overshoots and undershoots at the shock, not strong enough to activate the first order scheme. This drawback could be cured if one activates, in the MOOD method, the extrema detection procedure of \cite{clain} or \cite{Vilar}. When the limitation method is used, one can observe that the overshoot do not exist any more, while the undershoot are less important but still existing.
\begin{figure}[H]
\begin{center}
\subfigure[p]
{\includegraphics[width=0.45\textwidth]{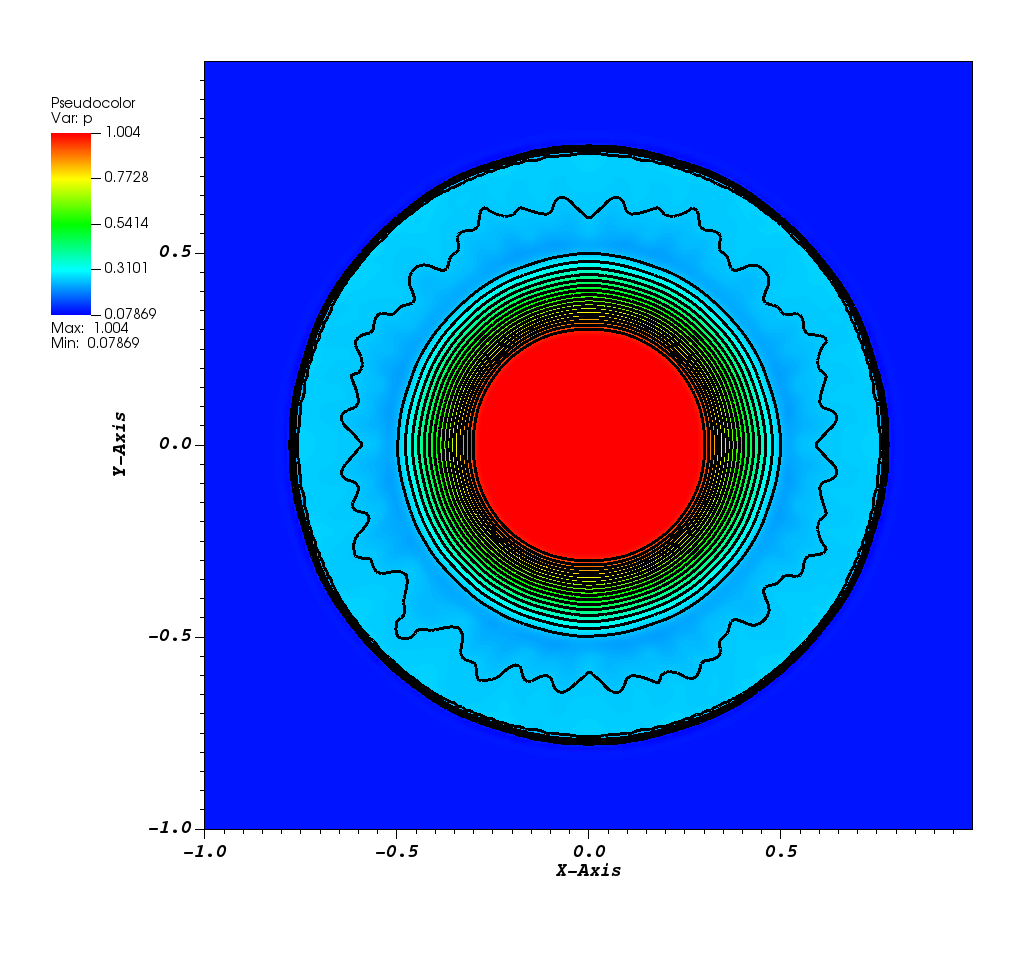}}
\subfigure[$\rho$]
{\includegraphics[width=0.45\textwidth]{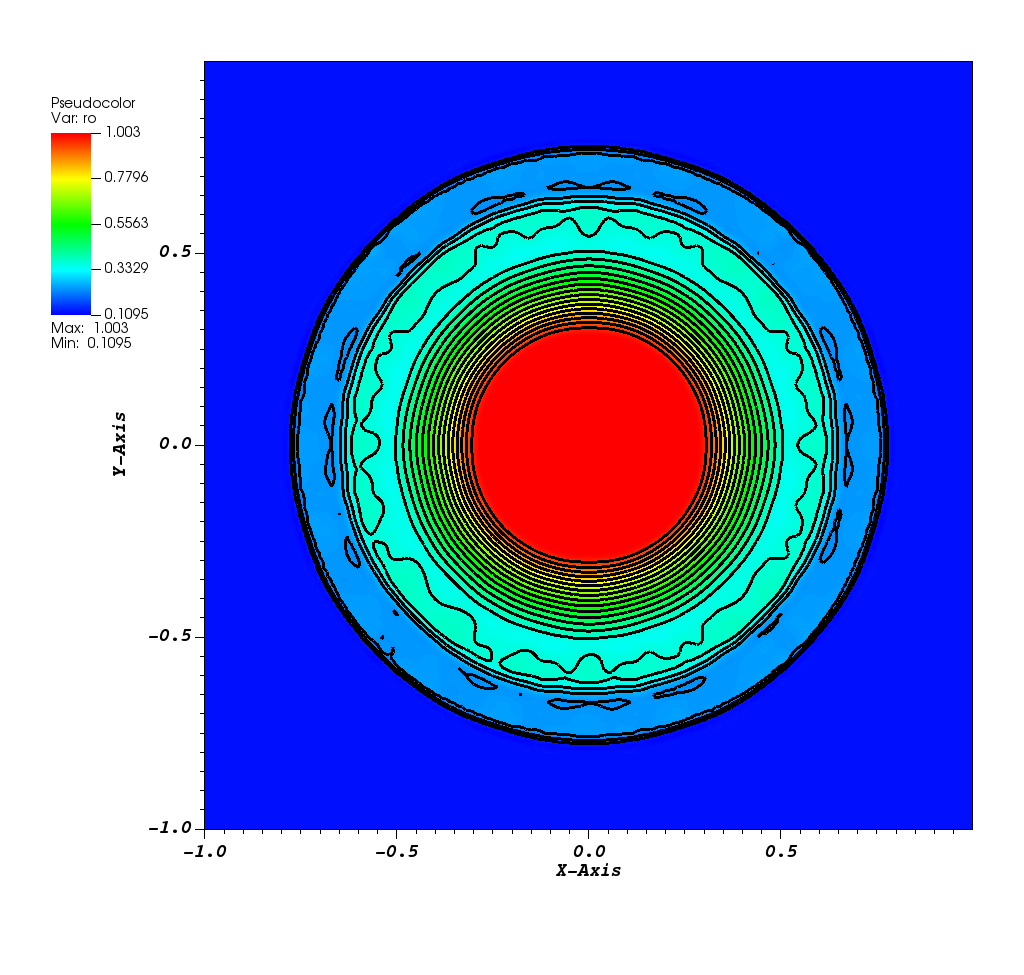}}
\subfigure[$\Vert \bv\Vert$]{\includegraphics[width=0.45\textwidth]{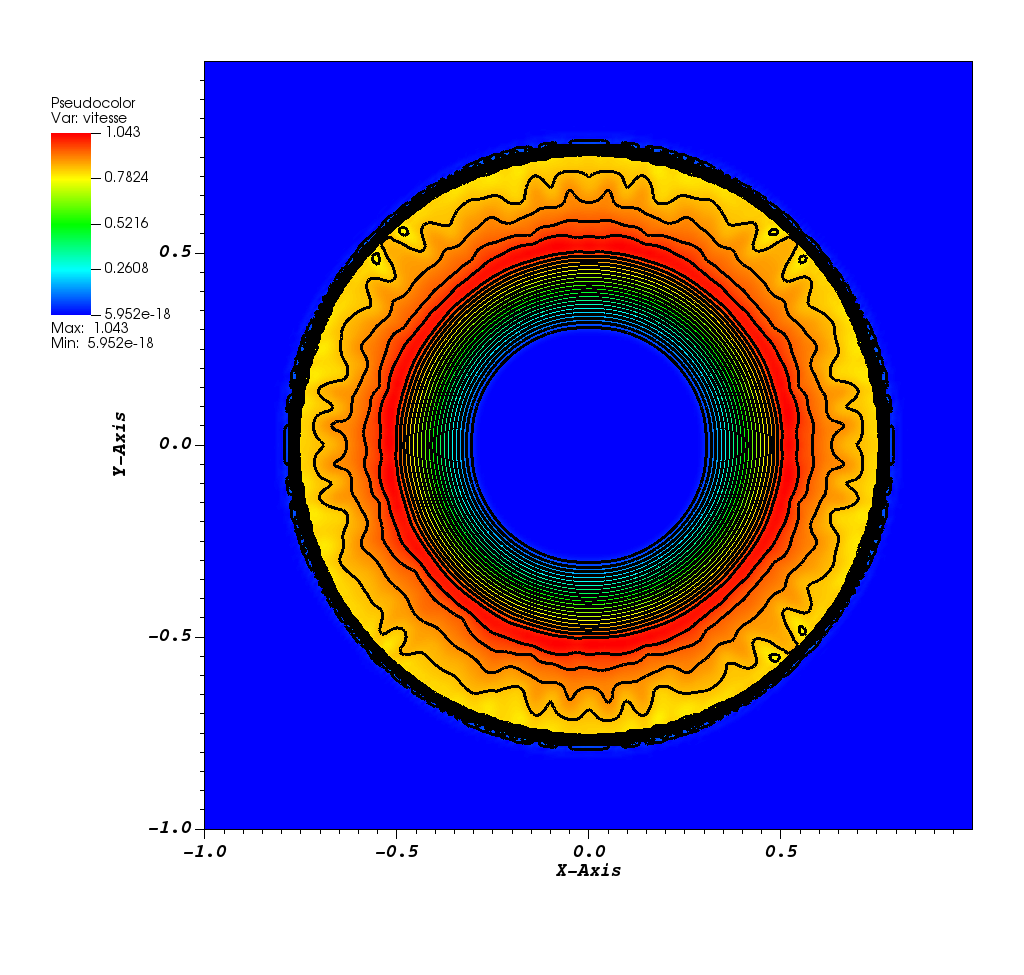}}
\end{center}
\caption{Sod problem, $T=0.16$ on a $200\times 200$ mesh with the limitation method.}
\label{fig:limit}
\end{figure}

\begin{figure}[H]
\begin{center}
\subfigure[p]
{\includegraphics[width=0.45\textwidth]{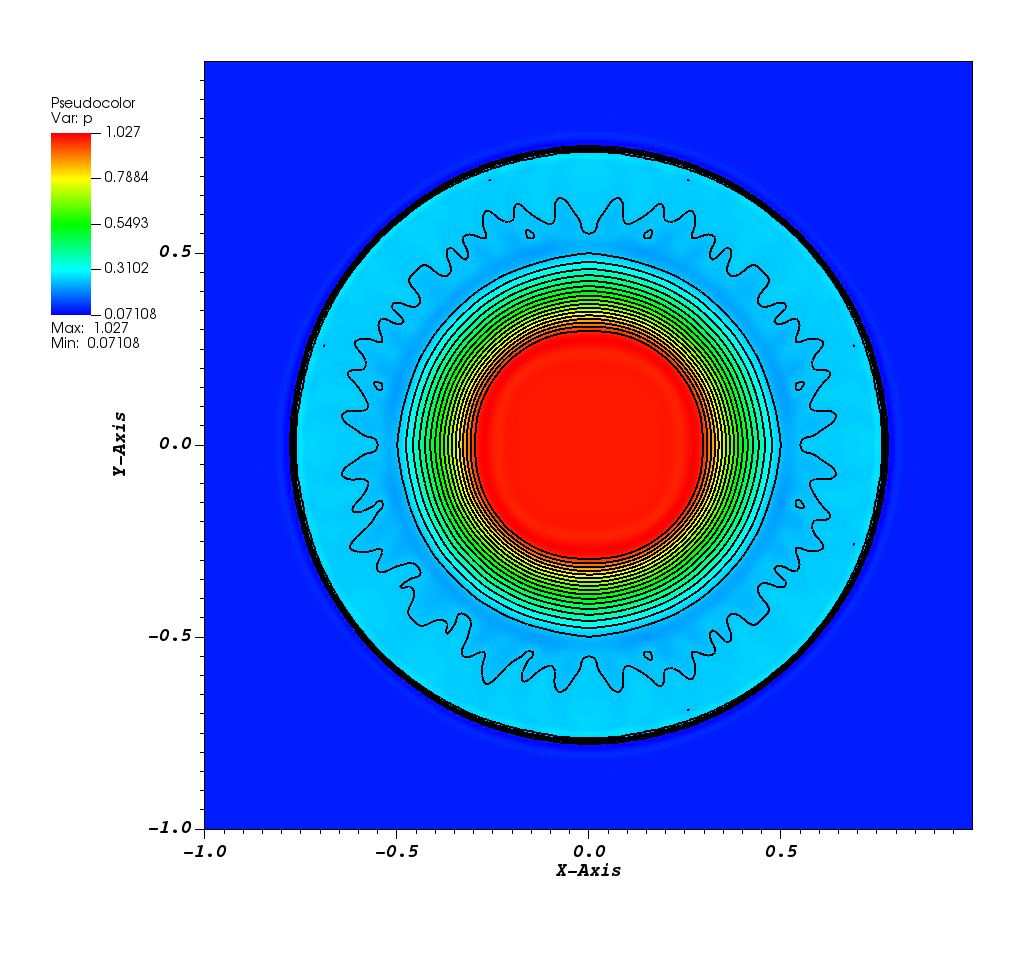}}
\subfigure[$\rho$]
{\includegraphics[width=0.45\textwidth]{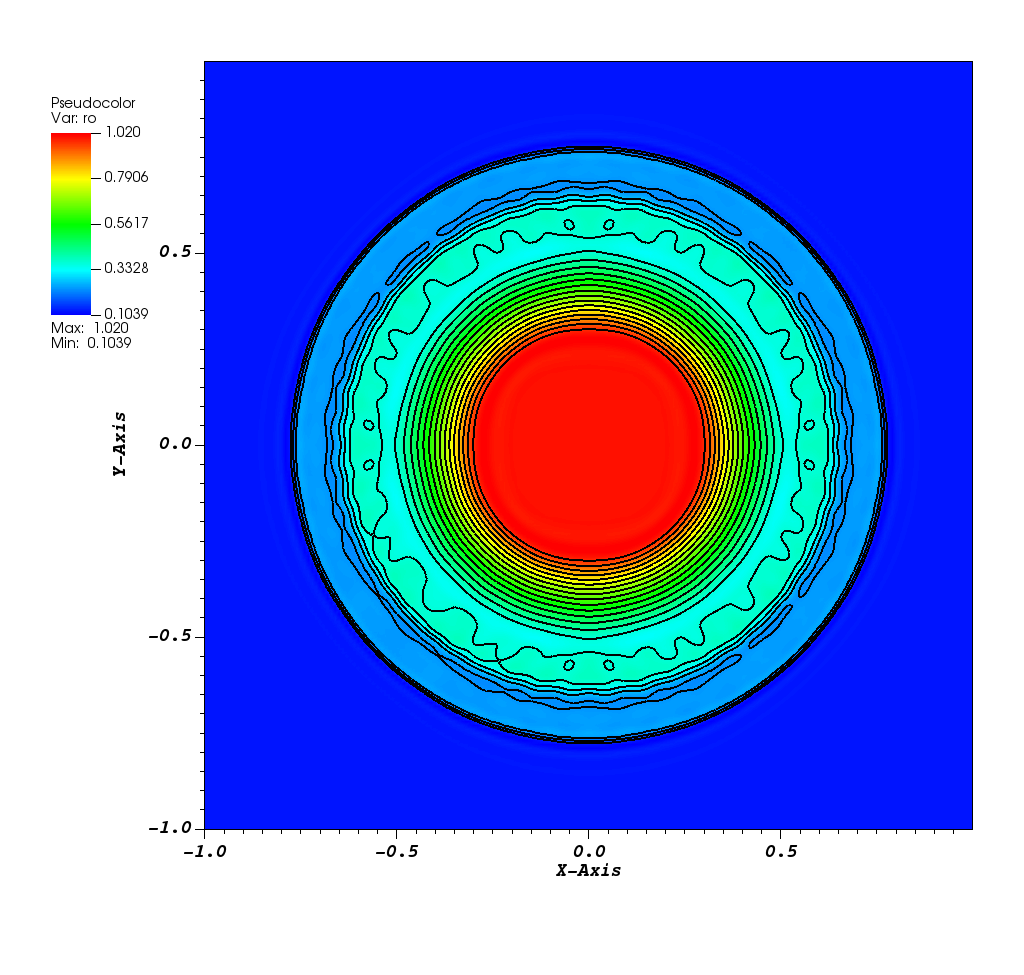}}
\subfigure[$\Vert \bv\Vert$]{\includegraphics[width=0.45\textwidth]{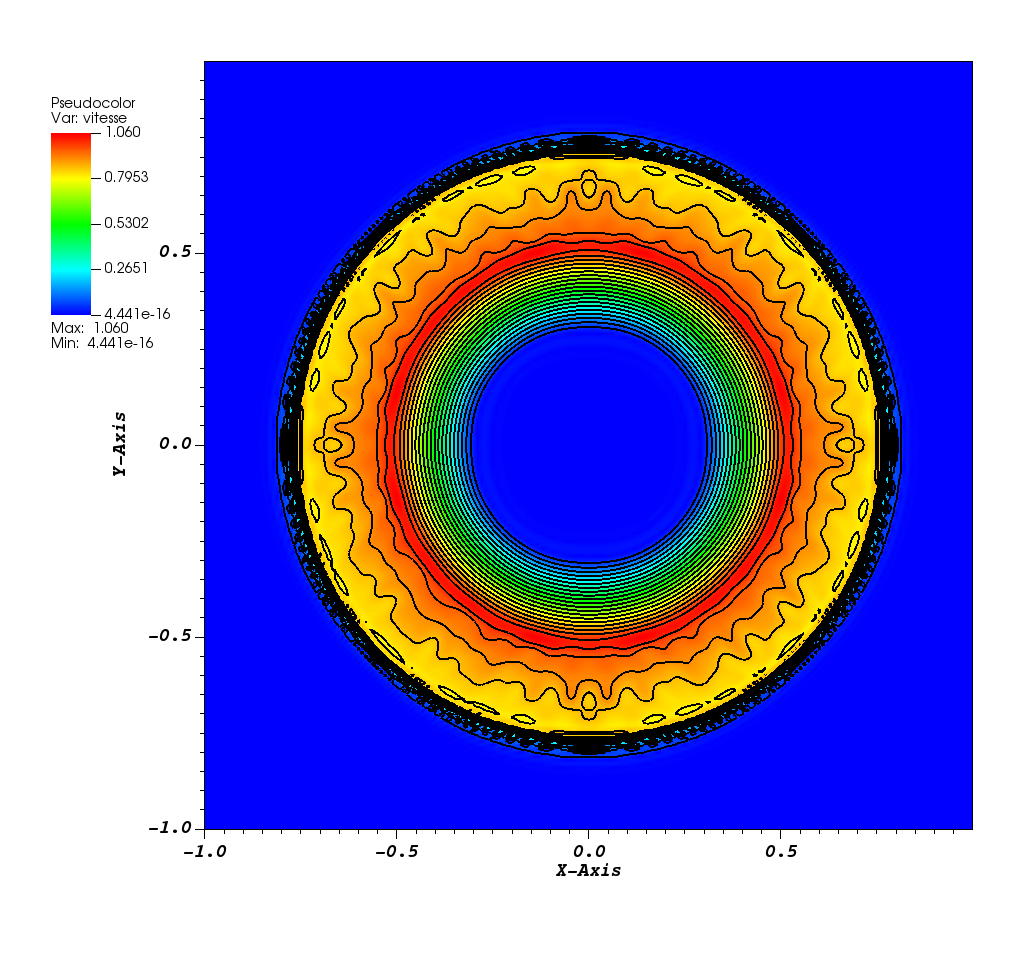}}
\end{center}
\caption{Sod problem, $T=0.16$ on a $200\times 200$ mesh with the MOOD method.}
\label{fig:mood}
\end{figure}
\subsubsection{Strong shock}
The problem is defined on $[-1.5,1.5]\times [-1.5,1.5]$ for $T=0.025$. We had to use the MOOD technique to get the results, the shocks are too strong.
\begin{equation}
\label{strong}
(\rho _0,v_{x,0},v_{y,0},p_0)=\left \{
\begin{array}{ll}
(1,0,0,1000)& \text{ if } r \leq 0.5\\
(1,0,0,1)& \text{else.}
\end{array}
\right .
\end{equation}
The pressure, density and norm of the velocity are displayed in Fig. \ref{StrongShock}, for the final time. The simulation is done with CFL$=1$ on a $200\times 200$ grid. On Fig. \ref{iso_lines}, we show the iso-lines of the density (mostly to localize the strong features of the solution) and the elements where the formal accuracy is dropped to first order. These flagged elements are moving in time, and are always localized around the discontinuities of the solution. In most cases, only a very few elements are flagged.
\begin{figure}[H]
\begin{center}
\subfigure[p]{\includegraphics[width=0.45\textwidth]{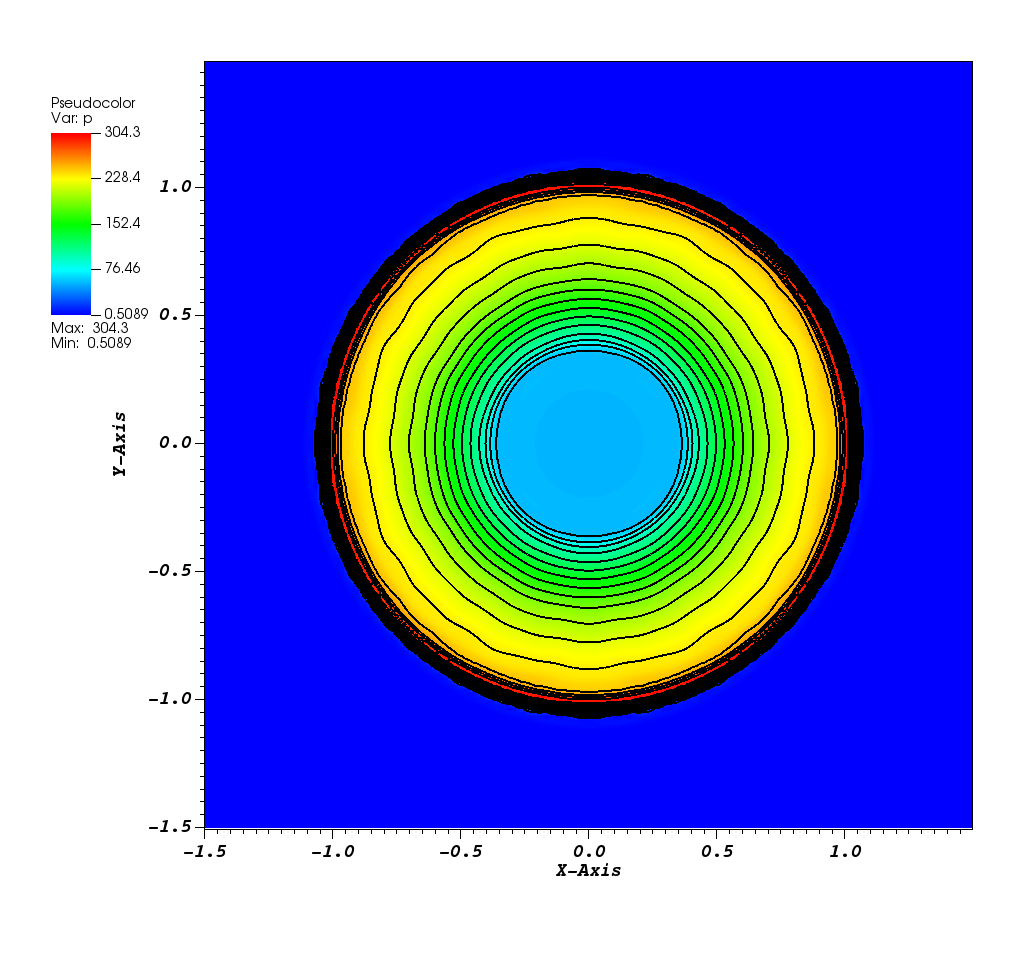}}
\subfigure[$\rho$]{\includegraphics[width=0.45\textwidth]{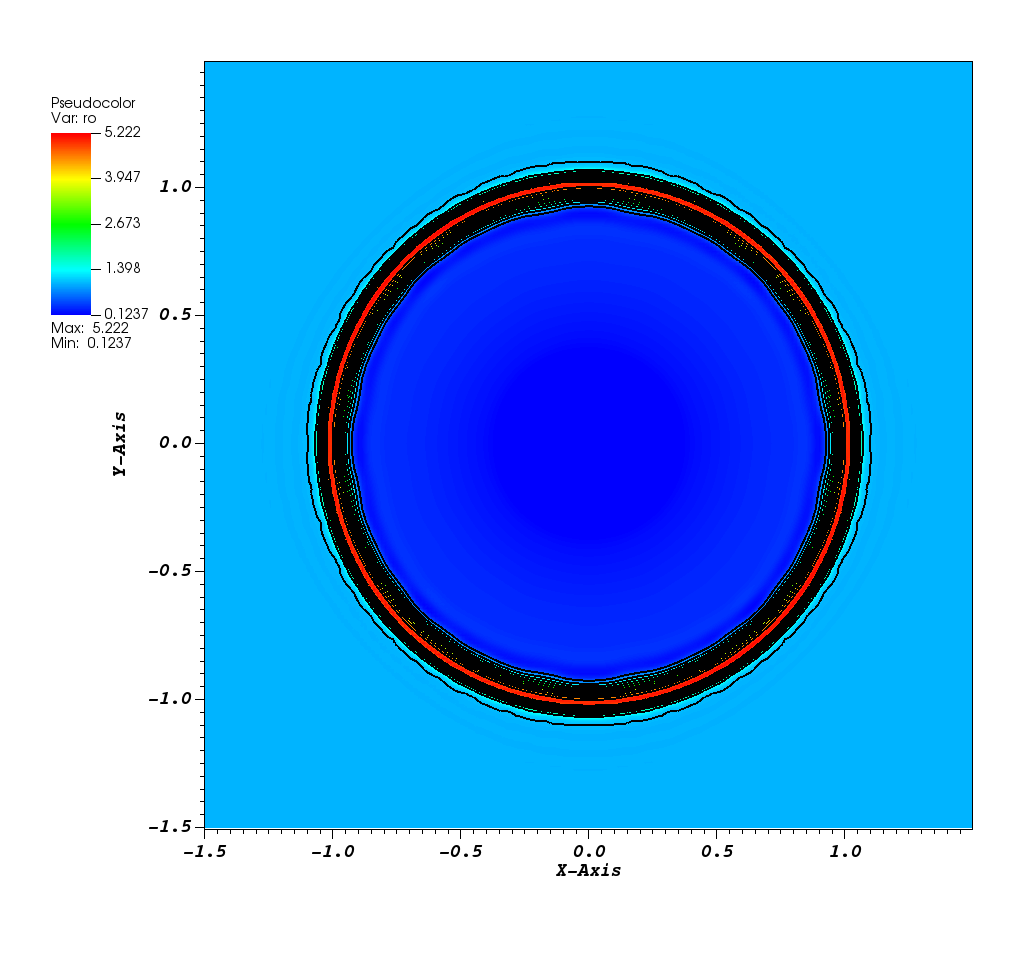}}
\subfigure[$\Vert\bv\Vert$]{\includegraphics[width=0.45\textwidth]{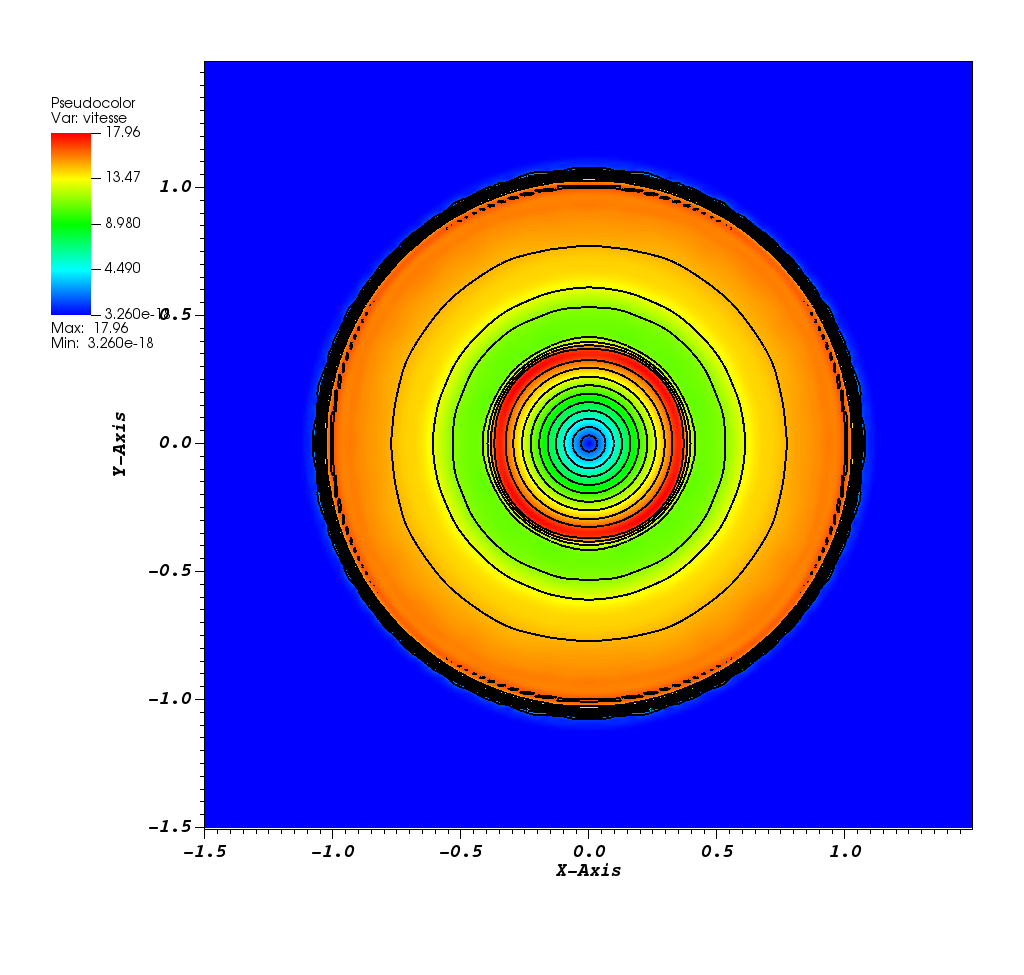}}
\subfigure[flag+$\rho$]{\includegraphics[width=0.45\textwidth]{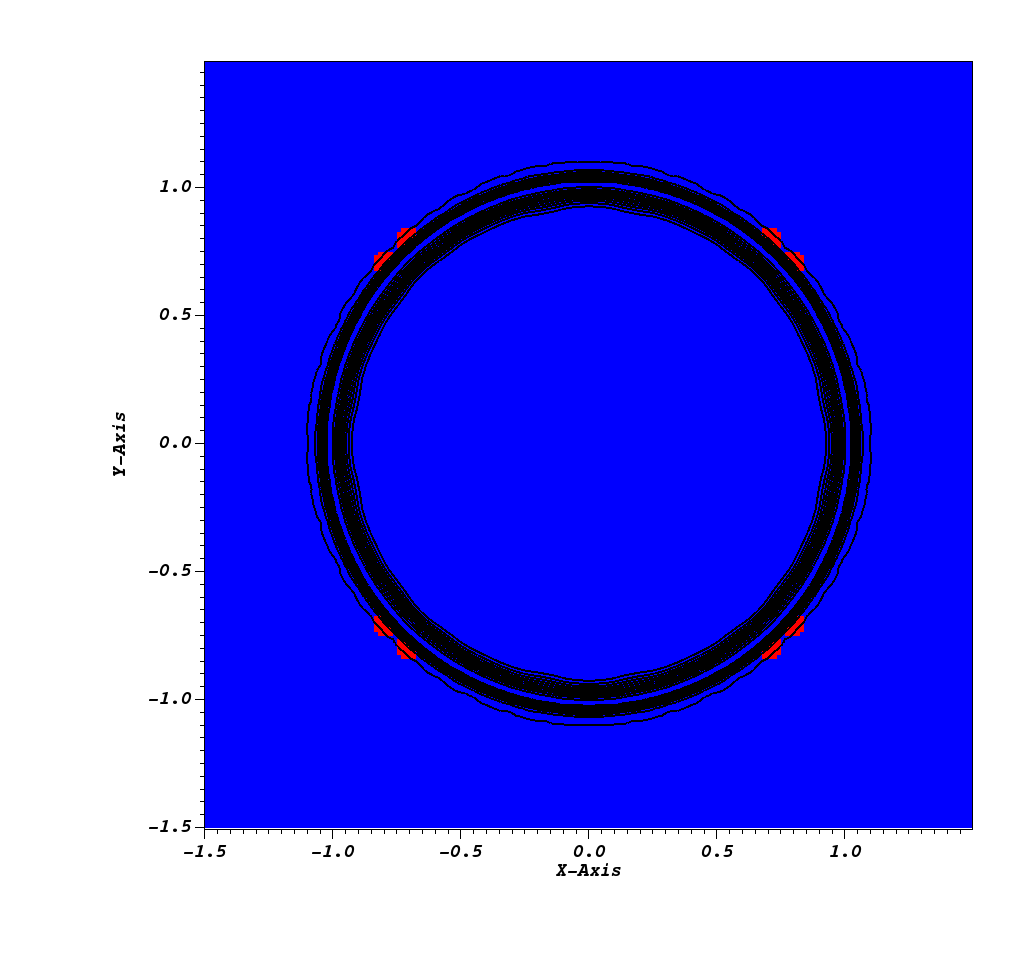}
\label{iso_lines}}
\end{center}
\caption{\label{strongshock}Result of case \ref{strong} on a $200\times200$ grid, $CFL=1$, space order: 4, time order: 4, MOOD, final time.}
\end{figure}
\section{Conclusion}
\label{kinetic:sec_conclusion}
The purpose of this work was primarily to extend a class of kinetic numerical methods that can run at least at CFL
one to the two dimensional case. These methods can handle in a simple manner hyperbolic problems, and in particular compressible fluid mechanics one.
Our methodology can be arbitrarily high order and can use CFL number larger or
equal to unity on regular Cartesian meshes. We have chosen the minimum number of waves, there are probably better solutions, and this will be the topic of further studies. These methods are not designed only for fluid mechanics, and other type of systems will be explored in the future. One interesting feature of this methods, working for CFL=1, is that the algebra for the streaming part of the algorithm can be made very efficient. This is an interesting feature.

\section*{Acknowledgments}
F.N.M has been funded by the SNF project  	200020\_204917 entitled "Structure preserving and fast methods for hyperbolic systems of conservation laws".
\bibliographystyle{unsrt}
\bibliography{main}

\begin{thebibliography}{10}

\bibitem{Abgrall}
R.~Abgrall and D.~Torlo.
\newblock Some preliminary results on a high order asymptotic preserving
  computationally explicit kinetic scheme.
\newblock {\em Communications in Mathematical Sciences}, 20(2):297--326, 2022.

\bibitem{Jin}
S.~Jin and Z.~Xin.
\newblock The relaxation schemes for systems of conservation laws in arbitrary
  space dimensions.
\newblock {\em Commun. Pure Appl. Math.}, 48(3):235–276, 1995.

\bibitem{Natalini}
R.~Natalini.
\newblock A discrete kinetic approximation of entropy solution to
  multi-dimensional scalar conservation laws.
\newblock {\em Journal of differential equations}, 148:292--317, 1998.

\bibitem{Bhatnagar}
P.~Bhatnagar, E.~Gross, and M.~Krook.
\newblock A model for collision processes in gases. \uppercase{I}. small
  amplitude processes in charged and neutral one-component systems.
\newblock {\em Phys. Rev.}, 94:511–525, 1954.

\bibitem{Cercignani}
C.~Cercignani.
\newblock The {B}oltzmann equation and its applications.
\newblock {\em Springer-Verlag, New York}, 1988.

\bibitem{Bouchut}
F.~Bouchut.
\newblock Construction of {BGK} models with a family of kinetic entropies for a
  given system of conservation laws.
\newblock {\em Journal of Statistical Physics}, 95(1/2), 1999.

\bibitem{Aregba}
D.~Aregba-Driollet and R.~Natalini.
\newblock Discrete kinetic schemes for multidimensional systems of conservation
  laws.
\newblock {\em SIAM Journal on Numerical Analysis}, 37(6):1973--2004, 2000.

\bibitem{Petra}
P.~Csomós and I.~Faragó.
\newblock Error analysis of the numerical solution of split differential
  equations.
\newblock {\em Math. Comput. Model.}, 48(7–8):1090–1106, 2008.

\bibitem{Schroll}
H.J. Schroll.
\newblock High resolution relaxed upwind schemes in gas dynamics.
\newblock {\em J. Sci. Comput.}, 17:599–607, 2002.

\bibitem{Banda}
M.~Banda and M.~Sead.
\newblock Relaxation weno schemes for multi-dimensional hyperbolic systems of
  conservation laws.
\newblock {\em Numer. Methods Partial. Differ. Equ.}, 23(5):1211–1234, 2007.

\bibitem{Lafitte}
P.~Lafitte, W.~Melis, and G.~Samaey.
\newblock A high-order relaxation method with projective integration for
  solving nonlinear systems of hyperbolic conservation laws.
\newblock {\em J. Comput. Phys.}, 340:1–25, 2017.

\bibitem{Coulette}
D.~Coulette, E.~Franck, P.~Helluy, M.~Mehrenberger, and L.~Navoret.
\newblock High-order implicit palindromic discontinuous galerkin method for
  kinetic-relaxation approximation.
\newblock {\em Comput. Fluids}, 190:485–502, 2019.

\bibitem{S.Jin}
S.~Jin.
\newblock Efficient asymptotic-preserving (ap) schemes for some multiscale
  kinetic equations.
\newblock {\em SIAM J. Sci. Comput.}, 21:441–454, 1999.

\bibitem{Boscarino}
S.~Boscarino, L.~Pareschi, and G.~Russo.
\newblock Implicit–explicit runge–kutta schemes for hyperbolic systems and
  kinetic equations in the diffusion limit.
\newblock {\em SIAM J. Sci. Comput.}, 35(1):22–51, 2013.

\bibitem{Sebastiano}
S.~Boscarino and G.~Russo.
\newblock On a class of uniformly accurate imex runge–kutta schemes and
  applications to hyperbolic systems with relaxation.
\newblock {\em SIAM J. Sci. Comput.}, 31(3):1926–1945, 2009.

\bibitem{Pareschi}
G.~Dimarco and L.~Pareschi.
\newblock Asymptotic-preserving implicit–explicit runge–kutta methods for
  nonlinear kinetic equations.
\newblock {\em SIAM J. Numer. Anal.}, 51(2):1064–1087, 2013.

\bibitem{Filbet}
F.~Filbet and S.~Jin.
\newblock A class of asymptotic-preserving schemes for kinetic equations and
  related problems with stiff sources.
\newblock {\em J. Comput. Phys.}, 229(20):7625–7648, 2010.

\bibitem{Remi}
R.~Abgrall.
\newblock High order schemes for hyperbolic problems using globally continuous
  approximation and avoiding mass matrices.
\newblock {\em J. Sci. Comput.}, 73(2):461–494, 2017.

\bibitem{AbgrallRD}
R{\'e}mi Abgrall.
\newblock Some remarks about conservation for residual distribution schemes.
\newblock {\em Comput. Methods Appl. Math.}, 18(3):327--351, 2018.

\bibitem{Hairer}
E.~Hairer and G.~Wanner.
\newblock Solving ordinary differential equations \uppercase{II}. stiff and
  differential-algebraic problems.
\newblock {\em Springer Series in Comput. Math., Springer-Verlag, Berlin}, 14,
  2010.

\bibitem{Iserles}
A.~Iserles.
\newblock Order stars and saturation theorem for first-order hyperbolics.
\newblock {\em IMA J. Numer. Anal.}, 2:49--61, 1982.

\bibitem{sweby}
P.~K. Sweby.
\newblock High resolution schemes using flux limiters for hyperbolic
  conservation laws.
\newblock {\em SIAM J. Numer. Anal.}, 21:995--1011, 1984.

\bibitem{Leveque}
Randall~J. LeVeque.
\newblock Wave propagation algorithms for multidimensional hyperbolic systems.
\newblock {\em J. Comput. Phys.}, 131(2):327--353, 1997.

\bibitem{Yee}
H.~C. Yee, R.~F. Warming, and Ami Harten.
\newblock On a class of {TVD} schemes for gas dynamic calculations.
\newblock Numerical methods for the {Euler} equations of fluid dynamics,
  {Proc}. {INRIA} {Workshop}, {Rocquencourt}/{France} 1983, 84-107 (1985).,
  1985.

\bibitem{clain}
S.~Diot, R.~Loub{\`e}re, and S.~Clain.
\newblock The multidimensional optimal order detection method in the
  three-dimensional case: very high-order finite volume method for hyperbolic
  systems.
\newblock {\em Int. J. Numer. Methods Fluids}, 73(4):362--392, 2013.

\bibitem{Vilar}
Fran{\c{c}}ois Vilar.
\newblock \emph{A posteriori} correction of high-order discontinuous {Galerkin}
  scheme through subcell finite volume formulation and flux reconstruction.
\newblock {\em J. Comput. Phys.}, 387:245--279, 2019.

\end{thebibliography}
\end{document}